\documentclass[a4paper,12pt]{article}
\usepackage{amsmath,amsthm,amsfonts,amssymb,bbm}
\usepackage{graphicx,psfrag,subfigure}
\usepackage{cite}
\usepackage{hyperref}
\usepackage[hmargin=1in,vmargin=1in]{geometry}
\hypersetup{colorlinks=true}

\renewcommand{\d}{\mathrm d}

\newcommand{\R}{\mathbb R}
\newcommand{\Z}{\mathbb Z}
\newcommand{\wt}{\widetilde}

\newcommand{\ol}{\overline}

\renewcommand{\Re}{\operatorname{Re}}

\newcommand{\sgn}{\operatorname{sgn}}

\newcommand{\C}{\mathcal C}
\newcommand{\D}{\mathcal D}

\newcommand{\g}{\mathrm g}
\renewcommand{\c}{\mathrm c}
\newcommand{\p}{\mathrm p}
\renewcommand{\a}{\mathrm a}
\newcommand{\gap}{\mathrm{gap}}

\newcommand{\id}{\mathbbm{1}}
\renewcommand{\O}{\mathcal{O}}
\renewcommand{\P}{\mathbf P}
\newcommand{\E}{\mathbf E}

\newcommand{\ami}{a_{\min}}
\newcommand{\ama}{a_{\max}}

\newtheorem{proposition}{Proposition}[section]
\newtheorem{theorem}[proposition]{Theorem}
\newtheorem{lemma}[proposition]{Lemma}

\theoremstyle{definition}
\newtheorem*{remark}{Remark}
\newtheorem*{remarks}{Remarks}

\numberwithin{equation}{section}

\author{B\'alint Vet\H o
\thanks{Department of Stochastics, Institute of Mathematics,
Budapest University of Technology and Economics, M\H uegyetem rkp.\ 3., H-1111 Budapest, Hungary. E-mail: {\tt vetob@math.bme.hu}.}
\thanks{MTA--BME Stochastics Research Group, M\H uegyetem rkp.\ 3., H-1111 Budapest, Hungary.}}

\title{Asymptotic fluctuations of geometric $q$-TASEP, geometric $q$-PushTASEP and $q$-PushASEP}

\date{}

\begin{document}

\maketitle
\begin{abstract}
We investigate the asymptotic fluctuation of three interacting particle systems: the geometric $q$-TASEP, the geometric $q$-PushTASEP and the $q$-PushASEP.
We prove that the rescaled particle position converges to the GUE Tracy--Widom distribution in the homogeneous case.
If the jump rates of the first finitely many particles are perturbed in the first two models,
we obtain that the limiting fluctuations are governed by the Baik--Ben Arous--P\'ech\'e distribution and that of the top eigenvalue of finite GUE matrices.

\end{abstract}

\section{Introduction}

The totally asymmetric simple exclusion process (TASEP) introduced in~\cite{Spi70} is the most well studied integrable model
which can be mapped into a surface growth model in the Kardar--Parisi--Zhang (KPZ) universality class.
The TASEP is an interacting particle system on the one-dimensional integer lattice $\Z$
where particles with vacant right neighbour jump to the right by one either according to independent Poisson clocks in continuous time
or independently with a certain probability in discrete time.
By using the underlying determinantal structure, the current fluctuations in TASEP are governed by the Airy processes~\cite{Jo03b,Sas05,BFPS06}.

The PushTASEP is a long-range version of TASEP introduced in~\cite{Lig80}
where the exclusion constraint preventing some particle jumps to happen is replaced by the pushing mechanism
which enforces an immediate jump of the particle at the target position further to the right if the target was occupied.
A two-dimensional stochastic particle system was introduced in~\cite{BF08}
whose two different one-dimensional (marginally Markovian) projections are TASEP and PushTASEP.

The $q$-TASEP is a one-parameter family of deformed TASEP models.
It was first introduced in~\cite{BC11} as a Markovian subsystem of the $q$-Whittaker 2d growth model
which is an interacting particle system in two space dimensions on Gelfand--Tsetlin patterns.
The parameter of $q$-TASEP is $q\in[0,1)$ and particles on $\Z$ jump to the right by one independently at rate $1-q^{\rm gap}$
where the gap is the number of consecutive vacant sites next to the particle on its right.
Due to the connection to the $q$-Whittaker process, a Fredholm determinant formula was obtained in~\cite{BCF12}
for the $q$-Laplace transform of the particle position in $q$-TASEP with step initial condition.
The asymptotic fluctuations of $q$-TASEP with step initial condition were shown to follow the GUE Tracy--Widom distribution in the homogeneous case~\cite{FV13}
using the formula in~\cite{BCF12} which confirms the KPZ universality and scaling theory conjectures~\cite{Spo12}.
A technical limitation of~\cite{FV13} was removed in~\cite{Bar14}
and in the presence of finitely many slower particles a Baik--Ben~Arous--P\'ech\'e transition is proved,
that is, the limiting fluctuations are BBP or the largest eigenvalue distribution of finite GUE matrices.


Two natural discrete time versions of $q$-TASEP were introduced in~\cite{BC15}, the geometric and Bernoulli discrete time $q$-TASEP
and Fredholm determinant expressions were proved for the $q$-Laplace transform of the particle positions in both systems.
These systems are expected to have the same scaling and asymptotic behaviour as $q$-TASEP, but no rigorous asymptotic analysis was performed so far.

The $q$-Whittaker 2d growth model has another Markovian projection which is the $q$-PushTASEP introduced and studied in~\cite{BP16}.
Its generalization with two-sided jumps is the $q$-PushASEP~\cite{CP15} which interpolates between $q$-TASEP and $q$-PushTASEP
and it is the $q$-deformed version of the PushASEP~\cite{BF07}.
Contour integral formulas for relevant observables are proved
and a Fredholm determinant formula for the $q$-Laplace transform of the particle position variable in $q$-PushASEP is conjectured in~\cite{CP15}.
The conjectural formula was proved in~\cite{MP17} in the framework of a new two-dimensional discrete time dynamics
along with Fredholm determinantal expressions for the Bernoulli and geometric $q$-PushTASEP.
In the special case of one-sided jumps the $q$-PushTASEP and geometric $q$-PushTASEP formulas of~\cite{MP17} are proved to the one given in~\cite{BCFV14}.

As an exactly solvable stochastic vertex model TASEP also has a three-parameter generalization to the q-Hahn TASEP~\cite{Pov13,C14,Vet15}.
The analogous extension of PushTASEP outside the Macdonald hierarchy is the $q$-Hahn PushTASEP and it was introduced in~\cite{CMP19}.
The pushing system of the $q$-Hahn PushTASEP was proved to have a Markov duality similar to the continuous time $q$-PushTASEP~\cite{CP15}.
The four-parameter family of stochastic higher spin vertex models~\cite{CP16} further generalizes the $q$-Hahn TASEP.
Another related model is the $q$-Hahn asymmetric exclusion process for which the discontinuity of the particle density
and Tracy--Widom fluctuations of particle positions involving the position of the first particle are known for a certain choice of parameters~\cite{BC16}.

We investigate the rigorous asymptotic analysis of the geometric $q$-TASEP, geometric $q$-PushTASEP and $q$-PushASEP in the present paper.
We prove GUE Tracy--Widom fluctuations for the particle position in the homogeneous case,
see the first statements in Theorems~\ref{thm:gmain}, \ref{thm:pmain} and \ref{thm:amain}.
If the jump rates of the first finitely many particles in the geometric $q$-TASEP and geometric $q$-PushTASEP are perturbed,
we obtain that the limiting fluctuations are Baik--Ben Arous--P\'ech\'e and the top eigenvalue distribution of finite GUE matrices
in Theorems~\ref{thm:pmain} and \ref{thm:amain}.

Our asymptotic analysis uses the following Fredholm determinant formulas for the $q$-Laplace transform of particle positions in the three interacting particle systems.
The finite time formula for geometric $q$-TASEP is available in~\cite{BC15}.
Theorem 3.3 in~\cite{BCFV14} has two applications in the present paper.
As a result of \cite{MP17} this theorem provides our starting formula for the geometric $q$-PushTASEP.
On the other hand by a different choice of specializations Theorem 3.3 in~\cite{BCFV14} gives a continuous time $q$-PushTASEP formula
which is generalized in~\cite{MP17} to the finite time formula for the $q$-PushASEP (where two-sided jumps are allowed) which we use.

The asymptotic results in this paper are formally similar to those in~\cite{FV13,Bar14,Vet15,BC16},
but their arguments are not directly applicable for the models studied in this paper.
The $q$-Hahn TASEP specializes to the geometric $q$-TASEP by setting the parameter $\nu=0$, but this case is not covered by the analysis in~\cite{Vet15}.
The main technical challenge in our models is the right choice of the integration contours and the proof of their steep descent property.
In particular compared to the analysis in~\cite{Vet15} it is not enough here to write the derivative of relevant functions along their steep descent contours
in terms of the function $h(b,s)$ given in~\eqref{defh} for the geometric $q$-PushTASEP and $q$-PushASEP.
Instead we work with a description using the function $e(b,s)$ given in~\eqref{defe}
which captures the behaviour of the derivatives along the contours around the points antipodal to the double critical point.
Furthermore we prove certain inequalities between various terms which appear in the derivative along the integration contours
by applying a new stochastic dominance argument which is formulated in Lemma~\ref{lemma:dominance}.
The technical conditions imposed in our theorems mostly come from the restriction that certain poles of the integrand have to be avoided
while deforming the original integration contours to the steep descent ones.

Based on the correspondence to certain random partitions under the half-space $q$-Whittaker measure
explicit $q$-Laplace transform formulas were obtained in~\cite{BBC20} for the particle position in the new interacting particle systems:
the geometric $q$-TASEP with activation and the geometric $q$-PushTASEP with particle creation.
These formulas are not Fredholm determinants however they resemble some Fredholm determinant expansions.
It turns out that the functions which govern the main contribution in the exponent of the integrals in these formulas
are exactly the same ones as in their full-space counterpart models which are studied in the present paper.
Hence the steep descent properties proved here about these functions along certain contours
heuristically imply the same limiting fluctuations also for the half-space models.
In order to make this argument rigorous, a dominated convergence argument is needed
which involves bounding the cross terms in a summable way along certain contours.
Proving this domination might use ideas from~\cite{D20}, but it does not seem to be straightforward.

The Baik--Ben Arous--P\'ech\'e limits proved in the present paper can possibly have a primary importance in showing convergence to the KPZ fixed point.
The KPZ fixed point is the expected space-time limit process for a wide family of models in the KPZ universality class.
It was constructed in~\cite{MQR17} and it was proved that the rescaled height function of TASEP converges to the KPZ fixed point.
It turned out that in certain models the question of convergence to the KPZ fixed point can be reduced to proving the convergence to the BBP distribution~\cite{Vir20}.
It is not obvious how the result of~\cite{Vir20} can be applied to the models of the present paper,
but the convergence to the BBP distribution in these models is certainly relevant for this reason as well.

Further parts of the paper are organized as follows.
We define geometric $q$-TASEP, geometric $q$-PushTASEP and $q$-PushASEP in Section~\ref{s:results}
and we state our main theorems about the asymptotic fluctuations in these models.
We prove the limit theorem for geometric $q$-TASEP in Section~\ref{s:g}, that for geometric $q$-PushTASEP in Section~\ref{s:p}
and the one for $q$-PushASEP in Section~\ref{s:a}.
The proofs of the steep descent properties of certain integration contours along with further technical statements are postponed to Section~\ref{s:steep}.

\paragraph{Acknowledgments.}
The author is grateful to Ivan Corwin and Leonid Petrov for discussions about explicit formulas for particle systems of this paper.
He thanks for valuable anonymous referee comments.
The work of the author was supported by the NKFI (National Research, Development and Innovation Office) grant FK123962,
by the Bolyai Research Scholarship of the Hungarian Academy of Sciences
and by the \'UNKP--21--5 New National Excellence Program of the Ministry for Innovation and Technology
from the source of the National Research, Development and Innovation Fund.

\section{Models and main results}
\label{s:results}

In order to introduce our models we define the $q$-Pochhammer symbol as
\begin{equation}\label{defqPochhammer}
(a;q)_n=\prod_{i=0}^{n-1}(1-aq^i),\qquad(a;q)_\infty=\prod_{i=0}^\infty(1-aq^i).
\end{equation}
Let $p_{m,\alpha}(j)$ denote the $q$-deformation of the truncated geometric distribution for $m=0,1,2,\dots$ and $\alpha\in(0,1)$, namely let
\begin{equation}\label{deftruncp}
p_{m,\alpha}(j)=\alpha^j(\alpha;q)_{m-j}\frac{(q;q)_m}{(q;q)_{m-j}(q;q)_j}\qquad\mbox{for}\quad j=0,1,\dots,m
\end{equation}
and let
\begin{equation}\label{deftruncpinfty}
p_{\infty,\alpha}(j)=\alpha^j(\alpha;q)_\infty\frac1{(q;q)_j}\qquad\mbox{for}\quad j=0,1,2,\dots.
\end{equation}
These non-negative weights define probability distributions, i.e.\ sum to $1$ by Lemma 1.3 of~\cite{BC15}.

\subsection{Geometric $q$-TASEP}

The discrete time geometric $q$-TASEP was introduced in~\cite{BC15} as follows.
The particles are ordered as
\begin{equation}\label{orderingg}
\infty=x_0^\g(t)>x_1^\g(t)>x_2^\g(t)>\dots>x_n^\g(t)>\dots
\end{equation}
for all $t=0,1,2,\dots$
and we define the gap between consecutive particles as $\gap_i^\g(t)=x_{i-1}^\g(t)-x_i^\g(t)-1$ with $\gap_1^\g(t)=\infty$ for all time $t$.
The initial configuration is the step initial condition, that is $x_i^\g(0)=-i$ for $i=1,2,\dots$.

The jump dynamics of the geometric $q$-TASEP is the following.
Let $a_1,a_2,\dots>0$ be particle rate parameters and let $\alpha_1,\alpha_2,\dots\in(0,1)$ be time dependent jump parameters
so that they satisfy the condition $a_i\alpha_j<1$ for all $i,j\ge1$.
Then the system of particles evolves from its state at time $t$ according to the parallel update rule
\begin{equation}\label{dynamicsg}
\P\left(x_i^\g(t+1)=x_i^\g(t)+j\,\big|\,\gap_i^\g(t)\right)=p_{\gap_i^\g(t),a_i\alpha_{t+1}}(j)
\end{equation}
where the steps of updates are independent for all $i$ and $t$.
Notice that by the definition of the truncated $q$-geometric distribution \eqref{deftruncp}, the ordering of particles \eqref{orderingg} is preserved by the dynamics \eqref{dynamicsg}.

The parameters in the geometric $q$-TASEP are chosen as follows.
We fix $a\in(0,1)$ and we consider the time-homogeneous model, that is, we set the time dependent jump rates $\alpha_1=\alpha_2=\dots=a$.
For a fixed $m$, the first $m$ particle rates $a_1,\dots,a_m\in(0,a^{-1})$ are arbitrary and the remaining ones are all equal to $a$, that is, $a_{m+1}=a_{m+2}=\dots=a$.
Let $\ami=\min\{a_1,\dots,a_m\}$ denote the lowest rate among the first $m$ particles.
In order to avoid further technicalities, we may assume that the first $m$ rates are at most $a$.
If some of these rates exceed $a$ then the corresponding particles either escape to infinity at a larger speed than the bulk
or they follow a slower particle in front of them.
Hence the limiting rescaled positions of particles at late times which we consider are ultimately influenced by the bulk rate $a$
and by the lowest rate $\ami$.
Choosing the $a_i$ and $\alpha_j$ parameters to be equal in the homogeneous case is natural
in the point of view of the corresponding half-space model, the geometric $q$-TASEP with activation of~\cite{BBC20} where there is only one family of parameters.

The large scale behaviour of the geometric $q$-TASEP is described below.
We define
\begin{equation}\label{defkappag}
\kappa_\g=\kappa_\g(\theta)=\frac{\Psi_q'(\theta-\log_qa)}{\Psi_q'(\theta+\log_qa)}
\end{equation}
for some $\theta\in(\log_qa,\infty)$ where $\Psi_q(z)$ is the $q$-digamma function.
The $q$-digamma function is given by $\Psi_q(z)=\frac{\partial}{\partial z}\ln\Gamma_q(z)$ where
\begin{equation}
\Gamma_q(z)=\frac{(q;q)_\infty}{(q^z;q)_\infty}(1-q)^{1-z}
\end{equation}
is the $q$-gamma function.
We consider the position of the $n$th particle after time $t=\kappa_\g n$.
We keep $\theta$ and hence $\kappa_\g$ fixed and we let $n\to\infty$.
The large scale behaviour of this particle position depends on $\ami$, namely the law of large numbers
\begin{equation}\label{LLNg}
\frac{x_n^\g(\kappa_\g n)}n\to\left\{\begin{array}{ll}f_\g &\mbox{if}\quad \theta\ge\log_q\ami,\\ g_\g &\mbox{if}\quad \theta\in(\log_qa,\log_q\ami)\end{array}\right.
\end{equation}
holds as $n\to\infty$ where
\begin{align}
f_\g&=\frac{\kappa_\g(\Psi_q(\theta+\log_qa)+\log(1-q))-(\Psi_q(\theta-\log_qa)+\log(1-q))}{\log q}-1,\label{deffg}\\
g_\g&=\frac{\kappa_\g(\Psi_q(\log_q\ami+\log_qa)+\log(1-q))-(\Psi_q(\log_q\ami-\log_qa)+\log(1-q))}{\log q}-1.\label{defgg}
\end{align}
The macroscopic shape of the particle positions is shown in Figure~\ref{fig:macro}.

\begin{figure}[t]
\centering
\includegraphics[width=200pt]{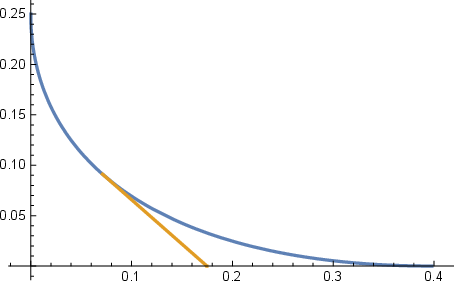}\qquad
\includegraphics[width=170pt]{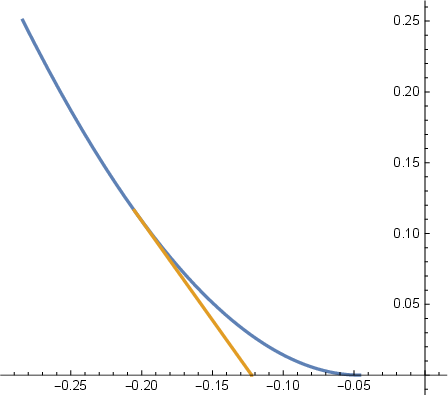}\\[4ex]
\includegraphics[width=230pt]{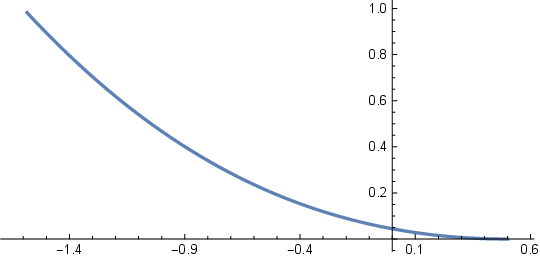}
\caption{The macroscopic shape of the particle positions in the three particle systems.
The vertical coordinate corresponds to particle label and the horizontal coordinate to the macroscopic position compared to the initial location.
The macroscopic position of the first particles is the intersection of the curves with the horizontal axis.
Top left: geometric $q$-TASEP, the parametric curve $((f_\g+1)/\kappa_\g,1/\kappa_\g)$ for $q=0.2,a=0.5$ in blue
(first particles at $(\Psi_q(2\log_qa)+\log(1-q))/\log q\simeq0.399$)
and the curve $((g_\g+1)/\kappa_\g,1/\kappa_\g)$ for $q=0.2,a=0.5,\ami=0.25$ in orange
(first particles at $(\Psi_q(\log_qa+\log_q\ami)+\log(1-q))/\log q\simeq0.175$).
Top right: geometric $q$-PushTASEP, the parametric curve $((-f_\p+1)/\kappa_\p,1/\kappa_\p)$ for $q=0.1,a=0.2$ in blue
(first particles at $-(\Psi_q(2\log_qa)+\log(1-q))/\log q\simeq-0.046$)
and the curve $((g_\p+1)/\kappa_\p,1/\kappa_\p)$ for $q=0.1,a=0.2,\ama=0.5$ in orange
(first particles at $-(\Psi_q(\log_qa+\log_q\ama)+\log(1-q))/\log q\simeq-0.122$).
Bottom: $q$-PushASEP, the parametric curve $((f_\a+1)/\kappa_\a,1/\kappa_\a)$ for $q=0.2,R=1,L=0.5$
(first particles at $R-L=0.5$).
\label{fig:macro}}
\end{figure}

The right end of the rarefaction fan, that is, the location of the first particles in the system
corresponds to the limit $\theta\to\log_qa$ when $\kappa_\g\to\infty$.
In this limit,
\begin{equation}\label{gfirstparticlespeed}
\lim_{\theta\to\log_q\!a}\frac{f_\g}{\kappa_\g}=\frac{\Psi_q(2\log_qa)+\log(1-q)}{\log q},\quad
\lim_{\theta\to\log_q\!a}\frac{g_\g}{\kappa_\g}=\frac{\Psi_q(\log_q\ami+\log_qa)+\log(1-q)}{\log q}.
\end{equation}
The first limit above is the speed of the first particle in the case when all particles have the same rate parameter $a$.
The second limit in \eqref{gfirstparticlespeed} is the asymptotic speed of the first particle with rate $\ami$.
These speeds can also be verified by computing the expected value of one jump of the first particle \eqref{deftruncpinfty} and with the $q$-binomial theorem
\begin{equation}
\sum_{k=0}^\infty z^k\frac{(b;q)_k}{(q;q)_k}=\frac{(bz;q)_\infty}{(z;q)_\infty}.
\end{equation}

We do not prove the law of large numbers \eqref{LLNg} directly since they follow from the main results below.
We consider the rescaled particle positions
\begin{equation}\label{defxietag}
\xi_n^\g=\frac{x_n^\g(\kappa_\g n)-f_\g n}{\chi_\g^{1/3}(\log q)^{-1}n^{1/3}},\qquad
\eta_n^\g=\frac{x_n^\g(\kappa_\g n)-g_\g n}{\sigma_\g^{1/2}(\log q)^{-1}n^{1/2}}
\end{equation}
where
\begin{equation}\label{defchisigmag}\begin{aligned}
\chi_\g&=\frac12\left(\kappa_\g\Psi_q''(\theta+\log_qa)-\Psi_q''(\theta-\log_qa)\right),\\
\sigma_\g&=\kappa_g\Psi_q'(\log_q\ami+\log_qa)-\Psi_q'(\log_q\ami-\log_qa)
\end{aligned}\end{equation}
are positive coefficients by Lemma~\ref{lemma:psimonotonicity}.

In accordance with the limit in the law of large numbers \eqref{LLNg}, there are three natural regimes depending on the relation of $\theta$ compared to $\log_q\ami$ which we address in the three parts of Theorem~\ref{thm:gmain}.
If $\theta>\log_q\ami$, then the rescaled position $\xi_n^\g$ converges to the Tracy--Widom distribution.
In the critical case $\theta=\log_q\ami+\O(n^{-1/3})$, the Baik--Ben Arous--P\'ech\'e distribution is the limit of $\xi_n^\g$.
If $\theta\in(\log_qa,\log_q\ami)$, then $\eta_n^\g$ converges to the top eigenvalue distribution of a finite GUE random matrix.
Note also that the scaling of $\xi_n^\g$ is $n^{1/3}$ whereas that of $\eta_n^\g$ is $n^{1/2}$.

\begin{theorem}\label{thm:gmain}
Let $q\in(0,1)$ be fixed and choose
\begin{equation}\label{gainterval}
a^2\in\left[q,1-\frac q{(1-q)^2}\right].
\end{equation}
Let $\theta$ be such that
\begin{equation}\label{gthetacond}
0<\theta-\log_qa<\log_q\left(\frac{2q}{1+q}\right).
\end{equation}
Then the following limits hold for the rescaled particle positions in the geometric $q$-TASEP.
\begin{enumerate}
\item
If $\theta>\log_q\ami$, then
\begin{equation}
\P(\xi_n^\g<x)\to F_{\rm GUE}(x)
\end{equation}
as $n\to\infty$ for all $x\in\R$ where $F_{\rm GUE}$ is the GUE Tracy--Widom distribution function, see~\cite{TW94}.

\item
If the particle jump parameters of the first $m$ particles depend on $n$ as
\begin{equation}\label{gBBPai}
a_i=q^{\theta+b_i\chi_\g^{-1/3}n^{-1/3}}
\end{equation}
for $i=1,\dots,m$, then
\begin{equation}\label{gBBP}
\P(\xi_n^\g<x)\to F_{\rm BBP,\bf b}(x)
\end{equation}
as $n\to\infty$ for all $x\in\R$ where $F_{\rm BBP,\bf b}$ is the Baik--Ben Arous--P\'ech\'e distribution function with parameter ${\bf b}=(b_1,\dots,b_m)$,
see~\cite{BBP06}.

\item
Let $q,a\in(0,1)$ be chosen so that \eqref{gainterval} holds.
Suppose that $\ami$ is such that the condition \eqref{gthetacond} with $\theta$ replaced by $\log_q\ami$, that is,
\begin{equation}\label{galphacond}
\frac{2q}{1+q}<\frac\ami a<1
\end{equation}
is satisfied.
Let $k$ denote the multiplicity of $\ami$ in the set $\{a_1,a_2,\dots\}$.
Then there is an $\varepsilon>0$ so that for any $\theta\in(\log_q\ami-\varepsilon,\log_q\ami)$
\begin{equation}\label{gGaussian}
\P(\eta_n^\g<x)\to G_k(x)
\end{equation}
as $n\to\infty$ for all $x\in\R$ where $G_k$ is the distribution of the largest eigenvalue of a $k\times k$ GUE random matrix.
\end{enumerate}
\end{theorem}

The Baik--Ben Arous--P\'ech\'e distribution was first introduced in~\cite{BBP06} and its distribution function is given by the Fredholm determinant expression $F_{\rm BBP,\bf b}(x)=\det(\id-K_{\rm BBP,\bf b})_{L^2((x,\infty))}$ with the kernel
\begin{equation}\label{defKBBP}
K_{\rm BBP,\bf b}(u,v)=\frac1{(2\pi i)^2}\int_{e^{-2\pi i/3}\infty}^{e^{2\pi i/3}\infty}\d w
\int_{e^{-\pi i/3}\infty}^{e^{\pi i/3}\infty}\d z
\frac{e^{z^3/3-zv}}{e^{w^3/3-wu}}\frac1{z-w}\prod_{i=1}^m\frac{z-b_i}{w-b_i}
\end{equation}
where the contour for $w$ crosses the real line to the right of all the $b_i$s and the contours for $w$ and for $z$ do not intersect.
The distribution $G_k$ has a Fredholm determinant expression as $G_k(x)=\det(\id-H_k)_{L^2((x,\infty))}$ with the kernel
\begin{equation}\label{defHk}
H_k(u,v)=\frac1{(2\pi i)^2}\int_{e^{-5\pi i/6}\infty}^{e^{5\pi i/6}\infty}\d w
\int_{e^{-\pi i/3}\infty}^{e^{\pi i/3}\infty}\d z\frac{e^{z^2/2-zv}}{e^{w^2/2-wu}}\frac1{z-w}\left(\frac zw\right)^k
\end{equation}
where the integration contours cross the real axis to the right of $0$ and they do not intersect each other, see e.g.~\cite{BH97,BK05}.

\begin{remarks}
\begin{enumerate}
\item
The conditions \eqref{gainterval} and \eqref{gthetacond} are technical.
The upper bound in \eqref{gthetacond} is the analogue of (2.16) in~\cite{Vet15} and it comes from the fact
that the poles of the sine function in the denominator of the kernel $K^\g_z$ in \eqref{defKgz} have to be avoided
by the integration contour for which the steep descent property is verified.

\item
By the presence of the $q$-Pochhammer symbols $(w/a_i;q)_\infty$ in the denominator of the kernel $K^\g_z$ in \eqref{defKgz},
there are poles at $w=a_iq^{-k}$ for $k=1,2,\dots$ which have to remain outside of the integration contour.
Since $a_i\in[\ami,a]\subseteq[q^\theta,a]$ for $i=1,\dots,m$, these poles are outside of $\C_\theta$ if $2a-q^\theta<q^{\theta-1}$
which is equivalent to \eqref{gthetacond}.

\item
The origin of \eqref{gainterval} is less intuitive, but it is needed for our method to work
in the proof of Proposition~\ref{prop:gsteep} for the steep descent property of the integration contours.
The interval for $a$ in \eqref{gainterval} is non-empty only if $0<q\le q^*\simeq0.318$
where $q^*$ is the unique solution of the equation $q=1-q/(1-q)^2$ within $(0,1)$.
With these possible values of $q$, the upper bound in \eqref{gthetacond} is between $\log_{q^*}(2q^*/(1+q^*))\simeq0.636$ and $1$.

\item
The $b_i=0$ for $i=1,\dots,m$ case in the second part of Theorem~\ref{thm:gmain} corresponds to choosing exactly $\theta=\log_q\ami$.
The $b_i\to-\infty$ limit means that we set the $i$th particle quicker than what could influence the limiting behaviour,
hence we get the Tracy--Widom fluctuations as in the first part of Theorem~\ref{thm:gmain}.
The convergence of the BBP distribution to the GUE Tracy--Widom distribution can also be seen directly from their kernels.
The extra factors $(z-b_i)/(w-b_i)\to1$ for all $w$ and $z$ pointwise as $b_i\to-\infty$.
By writing $(z-b_i)/(w-b_i)=(1-z/b_i)/(1-w/b_i)$ one can upper bound the integrand of the kernel in a way that
the Gaussian decay in the variables $w$ and $z$ along their respective contours and the Airy type decay $Ce^{-2/3(u^{3/2}+v^{3/2})}$ in $u$ and $v$
are enough to argue by dominated convergence.
This verifies the interchange of the limits $n\to\infty$ and $b_i\to-\infty$.

\item
Similarly to the case of $q$-TASEP~\cite{FV13} and to that of the $q$-Hahn TASEP~\cite{Vet15},
one can vary the time parameter in the definition of the rescaled particle position $\xi_n^\g$ on the scale $n^{2/3}$,
that is, one can consider the particle position at time $\kappa_\g n+cn^{2/3}$ for a new parameter $c\in\R$.
The Tracy--Widom limit remains unchanged as stated in Theorem~\ref{thm:gmain} after modifying the macroscopic position $f_\g n$ of the $n$th particle on the $n^{2/3}$ scale.
The effect on the Baik--Ben Arous--P\'ech\'e limit is that the coordinates $b_i$ of the parameter vector of the distribution on the right-hand side of \eqref{gBBP}
are replaced by $b_i-c\phi$ with some explicit $\phi>0$.
One expects by universality that in the homogeneous case the particles are non-trivially correlated on the scale $n^{2/3}$ with a limit being the Airy$_2$ process in the variable $c$.

\end{enumerate}
\end{remarks}

\subsection{Geometric $q$-PushTASEP}

The discrete time geometric $q$-PushTASEP first appears in~\cite{MP17}.
It is related to the $q$-PushTASEP~\cite{BP16} and the $q$-PushASEP~\cite{CP15} which were already introduced earlier as continuous time particle systems
that can be obtained as the limit of the discrete time geometric $q$-PushTASEP when the jump rates are let to $0$ and time is accelerated.

For any time $t=0,1,2,\dots$, the particles in the geometric $q$-PushTASEP are located at
\begin{equation}
0=x_0^\p(t)>x_1^\p(t)>x_2^\p(t)>\dots>x_n^\p(t)>\dots
\end{equation}
and their configuration evolves in discrete time steps.
We define $\gap_i^\p(t)=x_{i-1}^\p(t)-x_i^\p(t)-1$ for $i=1,2,\dots$ and
we assume that at time $t=0$, the system starts from the step initial condition $x_i(0)=-i$ for all $i=1,2,\dots$.

The dynamics of the geometric $q$-PushTASEP runs as follows.
Let $a_1,a_2,\dots>0$ be particle specific rate parameters and let $\alpha_1,\alpha_2,\dots$ be time dependent jump parameters
such that $a_i\alpha_j<1$ for all $i,j\ge1$.
Then positions of particles $x_n^\p(t)$ from time $t$ to $t+1$ are sequentially updated for $n=1,2,\dots$ as follows.
The particle $x_n^\p(t)$ jumps to the left to the new location
\begin{equation}\label{dynamicsp}
x_n^\p(t+1)=x_n^\p(t)-V_n(t)-W_n(t)
\end{equation}
where $\P(V_n(t)=k)=p_{\infty,a_n\alpha_t}(k)$ defined in \eqref{deftruncpinfty} and
\begin{equation}\label{Wdistr}
\P(W_n(t)=k)=(q^{\gap_n^\p(t)})^k(q^{\gap_n^\p(t)};q^{-1})_\infty
\frac{(q^{-1};q^{-1})_{c_{n-1}(t)}}{(q^{-1};q^{-1})_k(q^{-1};q^{-1})_{c_{n-1}(t)-k}}
\end{equation}
with $c_{n-1}(t)=x_{n-1}^\p(t)-x_{n-1}^\p(t+1)$ being the absolute size of the jump of particle $n-1$ in the same step from time $t$ to time $t+1$.
Note that the distribution of $W_n(t)$ is concentrated on the integers that are at least $c_{n-1}(t)$, hence the order of particles is preserved by the dynamics.
We remark that the probability on the right-hand side of \eqref{Wdistr} can formally be written as $p_{c_{n-1}(t),q^{\gap_n^\c(t)}}(k)$
defined in \eqref{deftruncp} with $q$ replaced by $q^{-1}$.

The relevant choice of the parameters for the geometric $q$-PushTASEP is the following.
We consider the time-homogeneous model, that is, we let $\alpha_1=\alpha_2=\dots=a$ for some fixed $a\in(0,1)$.
We let the first $m$ particle rates $a_1,\dots,a_m\in(0,a^{-1})$ to be arbitrary and $a_{m+1}=a_{m+2}=\dots=a$.
In the homogeneous case the rescaled particle positions have Tracy--Widom fluctuations,
but the Baik--Ben Arous--P\'ech\'e distribution and the top eigenvalue distribution of finite GUE matrices appear
if the first few particles have larger jump rate parameters which push further particles stronger to the right.
Therefore we assume that the first $m$ rates are at least $a$ and we denote by $\ama=\max\{a_1,\dots,a_m\}$ the highest rate.
This choice of parameters can also be followed for the geometric $q$-PushTASEP with particle creation in~\cite{BBC20}
which is the half-space analogue of geometric $q$-PushTASEP.

We define
\begin{equation}\label{defkappap}
\kappa_\p=\kappa_\p(\theta)=\frac{\Psi_q'(\log_qa-\theta)}{\Psi_q'(\log_qa+\theta)}
\end{equation}
for all $\theta\in(-\log_qa,\log_qa)$.
Then the law of large numbers
\begin{equation}\label{LLNp}
\frac{x_n^\p(\kappa_\p n)}n\to\left\{\begin{array}{ll}-f_\p &\mbox{if}\quad \theta\in(-\log_qa,\log_q\ama]\\
-g_\p &\mbox{if}\quad \theta\in(\log_q\ama,\log_qa)\end{array}\right.
\end{equation}
holds as $n\to\infty$ where
\begin{align}
f_\p&=\frac{\kappa_\p(\Psi_q(\log_qa+\theta)+\log(1-q))+\Psi_q(\log_qa-\theta)+\log(1-q)}{\log q}+1,\label{deffp}\\
g_\p&=\frac{\kappa_\p(\Psi_q(\log_qa+\log_q\ama)+\log(1-q))+\Psi_q(\log_qa-\log_q\ama)+\log(1-q)}{\log q}+1.\label{defgp}
\end{align}
See Figure~\ref{fig:macro} for the curve of the macroscopic particle positions in $q$-PushTASEP.

Next we consider the rescaled particle positions
\begin{equation}\label{defxietap}
\xi_n^\p=\frac{x_n^\p(\kappa_\p n)+f_\p n}{\chi_\p^{1/3}(\log q)^{-1}n^{1/3}},\qquad
\eta_n^\p=\frac{x_n^\p(\kappa_\p n)+g_\p n}{\sigma_\p^{1/2}(\log q)^{-1}n^{1/2}}
\end{equation}
where
\begin{equation}\label{defchisigmap}\begin{aligned}
\chi_\p&=-\frac12\left(\kappa_\p\Psi_q''(\log_qa+\theta)+\Psi_q''(\log_qa-\theta)\right),\\
\sigma_\p&=\kappa_\p\Psi_q'(\log_qa+\log_q\ama)-\Psi_q'(\log_qa-\log_q\ama)
\end{aligned}\end{equation}
are positive coefficients by Lemma~\ref{lemma:psimonotonicity}.

\begin{theorem}\label{thm:pmain}
Let $q,a\in(0,1)$ be fixed.
Let $\theta\in(-\log_qa,\log_qa)$ be such that
\begin{equation}\label{pcondpole}
\theta>\log_q\left(\frac{2a}{1+q}\right)
\end{equation}
and
\begin{equation}\label{pcondtech}
\frac{2q\left(a^2(1-q)^2+(q^\theta-a)^2\right)(2a-q^\theta)^2}{(2a-q^\theta-aq)^4}\le f_\p
\end{equation}
hold.

\begin{enumerate}
\item
If $\theta<\log_q\ama$, then for all $x\in\R$
\begin{equation}
\lim_{n\to\infty}\P(\xi_n^\p<x)=F_{\rm GUE}(x).
\end{equation}

\item
If the particle jump parameters depend of the first $m$ particles depend on $n$ as
\begin{equation}\label{pBBPai}
a_i=q^{\theta+b_i\chi_\p^{-1/3}n^{-1/3}}
\end{equation}
for $i=1,\dots,m$, then for all $x\in\R$
\begin{equation}\label{pBBP}
\lim_{n\to\infty}\P(\xi_n^\p<x)=F_{\rm BBP,\bf b}(x).
\end{equation}

\item
Let $q,a,\ama\in(0,1)$ be chosen so that the condition \eqref{pcondpole} with $\theta$ replaced by $\log_q\ama$,
that is,
\begin{equation}\label{palphacond}
\frac{\ama}a>\frac2{1+q}
\end{equation}
is satisfied.
Let $k$ denote the multiplicity of $\ama$ in the set $\{a_1,a_2,\dots\}$.
Then there is an $\varepsilon>0$ so that for any $\theta\in(\log_q\ama,\log_q\ama+\varepsilon)$ and for all $x\in\R$
\begin{equation}\label{pGaussian}
\lim_{n\to\infty}\P(\eta_n^\p<x)=G_k(x).
\end{equation}
\end{enumerate}
\end{theorem}

\begin{remark}
The condition \eqref{pcondpole} is essential for the steep descent contour to avoid the poles coming from the sine function in the denominator of the integrand.
We believe however that the second condition \eqref{pcondtech} is technical which can be checked numerically for any given choice of parameters.
Furthermore for any $q,a\in(0,1)$ fixed there is a $\theta^*<\log_qa$ such that \eqref{pcondtech} holds for all $\theta\in(\theta^*,\log_qa)$.
This is because for fixed $q,a\in(0,1)$ the left-hand side of \eqref{pcondtech} converges to $2q/(1-q)^2$ as $\theta\to\log_qa$
whereas $f_\p\to\infty$ as a function of $\theta$.
\end{remark}

\subsection{$q$-PushASEP}

The $q$-PushASEP is a continuous time particle system which was introduced in~\cite{CP15} as the $q$-deformation of the PushASEP of~\cite{BF07}.
We define the homogeneous model where the particle dependent jump rates are all chosen to be equal to one.
For any time $t=\ge0$, the particle positions are ordered
\begin{equation}
\infty=x_0^\a(t)>x_1^\a(t)>x_2^\a(t)>\dots>x_n^\a(t)>\dots
\end{equation}
and they evolve in continuous time.
We define $\gap_i^\a(t)=x_{i-1}^\a(t)-x_i^\a(t)-1$ for $i=1,2,\dots$ and
we assume that at time $t=0$, the system starts from the step initial condition $x_i(0)=-i$ for all $i=1,2,\dots$.

The dynamics of the $q$-PushASEP consists of jumps in two directions.
There are two fixed non-negative parameters $R,L$ which are not simultaneously zero.
Each particle at $x_i^\a(t)$ for $i=1,2,\dots$ jumps to the right by one at rate $R(1-q^{\gap_i^\a(t)})$ independently of other particles.
This rate vanishes when $\gap_i^\a(t)=0$ preventing jumps to positions already occupied.

Particles at $x_i^\a(t)$ for $i=1,2,\dots$ jump to the left by one at rate $L$ independently of other particles.
If any particle $x_j^\a(t)$ jumps to the left,
then it instantaneously pushes its left neighbour $x_{j+1}^\a(t)$ to the left by one with probability $q^{\gap_{j+1}^\a(t)}$.
If particle $x_{j+1}^\a(t)$ is pushed then it can also push particle $x_{j+2}^\a(t)$ to the left by one with probability $q^{\gap_{j+2}^\a(t)}$ and so on.
If $\gap_{j+1}^\a(t)=0$, then the push to the left happens with probability one,
that is particles moving to the left always push their immediate left neighbours further.

To describe the long time behaviour of $q$-PushASEP, we define
\begin{equation}\label{defkappaa}
\kappa_\a=\kappa_\a(\theta)=\frac{\Psi_q'(\theta)}{(\log q)^2(Rq^\theta+Lq^{-\theta})}
\end{equation}
for all $\theta>0$.
Then the law of large numbers
\begin{equation}\label{LLNa}
\frac{x_n^\a(\kappa_\a n)}n\to f_\a
\end{equation}
holds as $n\to\infty$ where
\begin{equation}\label{deffa}
f_\a=-\frac{\Psi_q(\theta)+\log(1-q)}{\log q}+\kappa_\a(Rq^\theta-Lq^{-\theta})-1.
\end{equation}
The curve of the macroscopic particle positions is shown in Figure~\ref{fig:macro}.
Next we consider the rescaled particle positions
\begin{equation}\label{defxia}
\xi_n^\a=\frac{x_n^\a(\kappa_\a n)-f_\a n}{\chi_\a^{1/3}(\log q)^{-1}n^{1/3}}
\end{equation}
where
\begin{equation}\label{defchia}
\chi_\a=\frac12\left(\kappa_\a(\log q)^3(Rq^\theta-Lq^{-\theta})-\Psi_q''(\theta)\right)
\end{equation}
is a positive coefficient by Lemma~\ref{lemma:psimonotonicity}.

\begin{theorem}\label{thm:amain}
Let $q\in(0,1)$ be fixed and let $R,L\ge0$ so that they are not simultaneously zero.
Then there is a $\theta^*>0$ such that for all $\theta\in(0,\theta^*)$ we have for all $x\in\R$ that
\begin{equation}
\lim_{n\to\infty}\P(\xi_n^\a<x)=F_{\rm GUE}(x).
\end{equation}
\end{theorem}

\begin{remark}
We expect that the inhomogeneous $q$-PushASEP admits a $q$-Laplace transform formula similar to the one in Theorem~\ref{thm:afinite}
with the factor $(w;q)_\infty^n$ replaced by $\prod_{i=1}^n(w/a_i;q)_\infty$ in the function $h(w)$.
Based on such a formula, the BBP transition can be similarly deduced as for the other two models in this paper
without extra further work in the BBP case since the same steep descent contours can be used.
To the best of our knowledge the inhomogeneous $q$-Laplace transform formula does not appear explicitly in the literature
because the conjecture was stated in~\cite{CP15} and proved in~\cite{MP17} in the homogeneous case for simplicity.
However it is natural to expect the inhomogeneous formula to hold as it appeared in~\cite{BC11} and~\cite{BCS12} in the $L=0$ case corresponding to $q$-TASEP
and in Theorem 3.3 of~\cite{BCFV14} with different contours in the $R=0$ case.
\end{remark}

\section{Asymptotics for the geometric $q$-TASEP}
\label{s:g}

This section is devoted to the proof of Theorem~\ref{thm:gmain} for the geometric $q$-TASEP.
The $q$-Laplace transform formula below for the geometric $q$-TASEP follows from Theorem 2.4 and Remark 2.5 in~\cite{BC15}.

\begin{theorem}[\!\!\cite{BC15}]\label{thm:gfinite}
For all $z\in\mathbb C\setminus\R_+$,
\begin{equation}\label{gfinite}
\E\left(\frac1{(zq^{x_n^\g(t)+n};q)_\infty}\right)=\det\left(\id-K^\g_z\right)_{L^2(C_{a_1,\dots,a_n})}
\end{equation}
where the kernel $K^\g_z$ in the Fredholm determinant above is given by
\begin{equation}\label{defKgz}
K^\g_z(w,w')=\frac1{2\pi i}\int_{1/2+i\R}\frac\pi{\sin(-\pi s)}(-z)^s\frac{h(q^sw)}{h(w)}\frac1{q^sw-w'}\,\d s
\end{equation}
with
\begin{equation}
h(w)=\frac{\prod_{i=1}^n(w/a_i;q)_\infty}{\prod_{j=1}^t(\alpha_jw;q)_\infty}.
\end{equation}
The integration contour $C_{a_1,\dots,a_n}$ in \eqref{gfinite} is a positively oriented curve which contains $a_1,\dots,a_n$ and no other poles of the kernel.
\end{theorem}

The integration contours in Theorem~\ref{thm:gfinite} are to be deformed in a way that Laplace's method of steepest descent can be applied.
We define the contours to be used as
\begin{equation}\label{defCD}
\C_\theta=\left\{W(s)=\log_q(a-(a-q^\theta)e^{is}),s\in(-\pi,\pi]\right\},\quad\D_\beta=\left\{Z(t)=\beta+i\frac t{\log q},t\in\R\right\}
\end{equation}
for any $\beta\in\R$.
The contour $\C_\theta$ is the image of the circle around $a$ that passes through $q^\theta$ under the map $\log_q$.
The contour $\D_\beta$ is a vertical line with real part $\beta$.

\subsection{GUE Tracy--Widom limit}

We prove the first part of Theorem~\ref{thm:gmain} for the geometric $q$-TASEP in this subsection.

\begin{lemma}\label{lemma:gLHS}
Let
\begin{equation}\label{zchoiceg}
z^\g_x=-q^{-(f_\g+1)n-\frac{\chi_\g^{1/3}}{\log q}xn^{1/3}}.
\end{equation}
Then for the left-hand side of \eqref{gfinite} on the time scale $t=\kappa_\g n$
\begin{equation}\label{qLaplaceg}
\lim_{n\to\infty}\E\left(\frac1{(z^\g_xq^{x_n^\g(\kappa_\g n)+n};q)_\infty}\right)=\lim_{n\to\infty}\P(\xi_n^\g<x).
\end{equation}
\end{lemma}

\begin{proof}
By the definition of $\xi_n^\g$ in \eqref{defxietag} and using \eqref{zchoiceg}, we can write
\begin{equation}
\E\left(\frac1{(z^\g_xq^{x_n^\g(\kappa_\g n)+n};q)_\infty}\right)
=\E\left(\frac1{(-q^{\chi_\g^{1/3}(\log q)^{-1}(\xi_n^\g-x)n^{1/3}};q)_\infty}\right).
\end{equation}
Since $\chi_\g>0$ (see Lemma~\ref{lemma:psimonotonicity}), the $q$-Laplace transform above has the same limit as the right-hand side of \eqref{qLaplaceg}
using the argument in Section 5 of~\cite{FV13}.
\end{proof}

We have to show that under the conditions of Theorem~\ref{thm:gmain} and with the scaling \eqref{zchoiceg},
\begin{equation}\label{gRHS}
\lim_{n\to\infty}\det(\id-K^\g_{z^\g_x})_{L^2(C_{a_1,\dots,a_n})}=F_{\rm GUE}(x).
\end{equation}
Lemma~\ref{lemma:gLHS} and \eqref{gRHS} imply the Tracy--Widom limit for the geometric $q$-TASEP in Theorem~\ref{thm:gmain}.
The limit \eqref{gRHS} follows from Propositions~\ref{prop:glocalization} and \ref{prop:gconvergence} as it is explained here.

First we apply our choice of parameters $\alpha_1=\alpha_2=\dots=a$ and $a_{m+1}=a_{m+2}=\dots=a$.
Then we perform the change of variables $w=q^W, w'=q^{W'}, s=Z-W$ in the Fredholm determinant on the right-hand side of \eqref{gfinite}.
The Fredholm determinant on the right-hand side of \eqref{gfinite} for the choice \eqref{zchoiceg} is equal to the Fredholm determinant of the rescaled kernel
\begin{multline}\label{Krescaledg}
q^W\log qK^\g_{z^\g_x}(q^W,q^{W'})\\
=\frac{q^W\log q}{2\pi i}\int_{\theta+i\R}\frac{\d Z}{q^Z-q^{W'}}\frac\pi{\sin(\pi(Z-W))}
\frac{e^{nf_0^\g(W)+n^{1/3}f_2^\g(W)}}{e^{nf_0^\g(Z)+n^{1/3}f_2^\g(Z)}}\frac{\Phi^\g(W)}{\Phi^\g(Z)}
\end{multline}
where
\begin{align}
f_0^\g(W)&=\kappa_\g\log(aq^W;q)_\infty-\log(q^W/a;q)_\infty+(f_\g+1)\log q\,W,\\
f_2^\g(W)&=\chi_\g^{1/3}xW
\end{align}
and
\begin{equation}\label{defPhig}
\Phi^\g(W)=\frac{(q^W/a;q)_\infty^m}{\prod_{i=1}^m(q^W/a_i;q)_\infty}.
\end{equation}
The integration contour for the kernel in \eqref{Krescaledg} for $W$ has to encircle $\log_qa$ and $\log_qa_i$ for all $i=1,\dots,m$
but no other singularities of the kernel.
We will choose the contour for $W$ to be $\C_\theta$ and the contour for $Z$ to be $\D_\theta$ defined in \eqref{defCD}.
This can be achieved by continuous deformation without crossing any poles of the integrand
which is checked in Step 1 in the proof of Proposition~\ref{prop:glocalization}.

Taylor expansion around $\theta$ yields that
\begin{equation}\label{f0gTaylor}
f_0^\g(W)=f_0^\g(\theta)-\frac{\chi_\g}3(W-\theta)^3+\O((W-\theta)^4).
\end{equation}

The following steep descent properties for the contours $\C_\theta$ and $\D_\theta$ are the key for the asymptotics of the geometric $q$-TASEP.

\begin{proposition}\label{prop:gsteep}
Assume that the conditions of Theorem~\ref{thm:gmain} hold.
\begin{enumerate}
\item
The contour $\C_\theta$ is of steep descent for the function $\Re(f_0^\g)$ in the sense that the function attains its maximum $q^\theta$
which corresponds to $s=0$, it increases for $s\in(-\pi,0)$ and it decreases for $s\in(0,\pi)$.
\item
The function $-\Re(f_0^\g)$ is periodic along $\D_\theta$ with period $2\pi|\log q|$.
The period symmetric about the real axis which corresponds to $t\in[-\pi|\log q|,\pi|\log q|]$ is of steep descent
for the function $-\Re(f_0^\g)$ in the following sense.
The function attains its maximum at $q^\theta$ which corresponds to $t=0$, it increases for $t\in(-\pi|\log q|,0)$ and it decreases for $t\in(0,\pi|\log q|)$.
\end{enumerate}
\end{proposition}

Proposition~\ref{prop:gsteep} is proved in Subsection~\ref{ss:gsteep}.
The rest of the argument can be done very similarly to Propositions 5.1--5.4 of~\cite{Vet15}.

We introduce the V-shaped contour
\begin{equation}\label{defVcont}
V_{\beta,\varphi}^\delta=\{\beta+e^{i\varphi\sgn(t)}|t|,t\in[-\delta,\delta]\}
\end{equation}
where $\beta\in\R$ is the tip of the V, $\varphi\in(0,\pi)$ is its half-angle and $\delta\in\R_+$.
Let
\begin{equation}
K^\g_{x,\delta}(W,W')
=\frac{q^W\log q}{2\pi i}\int_{V_{\theta,\varphi}^\delta}\frac{\d Z}{q^Z-q^{W'}}\frac\pi{\sin(\pi(Z-W))}
\frac{e^{nf_0^\g(W)+n^{1/3}f_2^\g(W)}}{e^{nf_0^\g(Z)+n^{1/3}f_2^\g(Z)}}\frac{\Phi^\g(Z)}{\Phi^\g(W)}
\end{equation}
where $W,W'\in V_{\theta,\pi-\varphi}^\delta$.

\begin{proposition}\label{prop:glocalization}
Fix $x\in\R$.
For any fixed $\varepsilon>0$ small enough, there are $\delta>0$, $\varphi\in(\pi/6,\pi/2)$ and $n_0$ such that for all $n>n_0$
\begin{equation}
\left|\det(\id-K^\g_{z_x})_{L^2(C_{a_1,\dots,a_n})}-\det(\id-K^\g_{x,\delta})_{L^2(V_{\theta,\pi-\varphi}^\delta)}\right|<\varepsilon.
\end{equation}
\end{proposition}

\begin{proof}[Proof of Proposition~\ref{prop:glocalization}]
We proceed in three steps as in the proof of Proposition 5.1 in~\cite{Vet15}.

\emph{Step 1: Contour deformation.}
We apply the Cauchy theorem to change the integration contours for \eqref{Krescaledg} to $W\in\C_\theta$ and $Z\in\D_\theta$.
What we have to check is that, under the conditions of Theorem~\ref{thm:gmain}, no singularity of the integrand is crossed during the deformation.
The poles coming from the sine in the denominator are avoided if \eqref{gthetacond} holds, because $\Re(W)\in[\log_q(2a-q^\theta),\theta]$ and $\Re(Z)=\theta$ along $W\in\C_\theta$ and $Z\in\D_\theta$, hence $\sin(\pi(Z-W))$ does not have a $0$ if $\theta-1<\log_q(2a-q^\theta)$ which is the upper bound in \eqref{gthetacond}.

On the other hand there are poles coming from the $q$-Pochhammer symbols in the denominator at $Z=0,-1,-2,\dots$ which are certainly avoided.
Further poles are at $W=\log_qa_i,\log_qa_i-1,\log_qa_i-2,\dots$ for $i=1,2,\dots$.
The first ones at $\log_qa_i$ are encircled by $\C_\theta$ since it was assumed that $\log_qa_i$ all lie in the interval $[\log_qa,\log_q\ami]$.
The remaining poles at $\log_qa_i-1,\log_qa_i-2,\dots$ are avoided since $\log_qa_i-1<\log_q(2a-q^\theta)$
which holds by \eqref{gthetacond} and by $\log_qa_i<\theta$.

\emph{Step 2: Localization to short contours.}
Once the integration contours are changed to $W\in\C_\theta$ and $Z\in\D_\theta$, the integrand is bounded uniformly
as we are away from the critical point $\theta$.
The factor $\Phi^\g(W)/\Phi^\g(Z)$ coming from the different jump parameters for the first $m$ particles is also bounded and it converges to $1$
as long as the minimal jump parameter $\ami>q^\theta$ is away from the critical point.
Due to the factor $\pi/\sin(\pi(Z-W))$, the kernel $K^\g_{z_x}$ has a logarithmic divergence in the neighbourhood of $s=0$
using the parametrization of $\C_\theta$, but this is also integrable.
By Proposition~\ref{prop:gsteep} the contours $\C_\theta$ and $\D_\theta$ are of steep descent.
Hence the contour $W\in\C_\theta$ can be localized to $W\in V_{\theta,\pi-\varphi}^\delta$ for some $\varphi\in(\pi/6,\pi/2)$ and $\delta>0$
and the integration over $Z\in\D_\theta$ can be changed to $Z\in V_{\theta,\pi/2}^\delta$ by making an error of order $\O(\exp(-c\delta^3n))$
in the same way as in the proof of Proposition 5.1 in~\cite{Vet15}.

\emph{Step 3: Deformation of short contours.}
The contour $Z\in V_{\theta,\pi/2}^\delta$ can be changed to $Z\in V_{\theta,\varphi}^\delta$ for some $\varphi\in(\pi/6,\pi/2)$
with and error of order $\O(\exp(-c\delta^3n))$ if $\delta$ is small enough.
This is because the main contribution for the $Z$ integral is coming from the factor $e^{-nf_0^\g(Z)}$
and the local behaviour of $f_0^\g$ around $\theta$ is governed by the Taylor expansion \eqref{f0gTaylor}.
See the proof of Proposition 5.1 in~\cite{Vet15} for further details.
\end{proof}

\begin{proposition}\label{prop:gconvergence}
For any $x\in\R$, $\delta>0$ and $\varphi\in(\pi/6,\pi/2)$ we have that
\begin{equation}\label{gconvergence}
\det(\id-K^\g_{x,\delta})_{L^2(V_{\theta,\pi-\varphi}^\delta)}\to F_{\rm GUE}(x)
\end{equation}
as $n\to\infty$.
\end{proposition}

\begin{proof}[Proof of Proposition~\ref{prop:gconvergence}]
The convergence \eqref{gconvergence} follows as in the proofs of Propositions~5.2--5.4 in~\cite{Vet15}.
In particular, after the change of variables $W=\theta+wn^{-1/3}$, $W'=\theta+w'n^{-1/3}$, $Z=\theta+zn^{-1/3}$ and by using the Taylor approximation \eqref{f0gTaylor}, the Fredholm determinant $\det(\id-K^\g_{x,\delta})$ differs by an error of order $\O(n^{-1/3})$ from $\det(\id-AB)$ where
\begin{equation}
A(w,\lambda)=e^{-\chi_\g w^3/3+(\chi_\g^{1/3}x+\lambda)w},\qquad
B(\lambda,w)=\frac1{2\pi i}\int\frac{\d z}{z-w}\,e^{\chi_\g z^3/3-(\chi_\g^{1/3}x+\lambda)z}.
\end{equation}
The factor $\Phi^\g(W)/\Phi^\g(Z)\to1$ and it does not appear in the limit because $\ami$ is away from $q^\theta$.
Using the identity $\det(\id-AB)=\det(\id-BA)$ and by rescaling by $\chi_\g^{1/3}$ yields \eqref{gconvergence}.
\end{proof}

\subsection{Baik--Ben Arous--P\'ech\'e limit}
\label{ss:BBP}

We prove the second part of Theorem~\ref{thm:gmain} about the Baik--Ben Arous--P\'ech\'e limit for the rescaled position $\xi_n^\g$ in the geometric $q$-TASEP.
We proceed in almost the same way as in the proof of the Tracy--Widom limit, the only difference is
the behaviour of the $\Phi^\g(W)/\Phi^\g(Z)$ factor in the kernel $K^\g_{x,\delta}$ in Proposition~\ref{prop:gconvergence}.
Hence the contour deformation and localization steps in Proposition~\ref{prop:glocalization} can be done identically.

After the change of variables $W=\theta+wn^{-1/3}$, $W'=\theta+w'n^{-1/3}$, $Z=\theta+zn^{-1/3}$ which corresponds to the first step in the proof of Proposition~\ref{prop:gconvergence}, we have $|wn^{-1/3}-\theta|<\delta$ and $|zn^{-1/3}-\theta|<\delta$ for some small $\delta>0$.
With the scaling of parameters \eqref{gBBPai} in the relevant factor of $\Phi^\g(W)$ in \eqref{defPhig} and with the change of variables above, we can write
\begin{equation}\label{Phioneterm}
(q^W/a_i;q)_\infty=\left(1-q^{(w-b_i\chi_\g^{-1/3})n^{-1/3}}\right)\prod_{k=1}^\infty\left(1-q^{k+(w-b_i\chi_\g^{-1/3})n^{-1/3}}\right).
\end{equation}
The infinite product on the right-hand side above is bounded on the integration contour cut off by the bound $|wn^{-1/3}-\theta|<\delta$
and it converges to the constant $(q;q)_\infty$ for any fixed $w$ along the localized contour as $n\to\infty$.
The first factor on the right-hand side of \eqref{Phioneterm} is asymptotically equal to $-\log q(w-b_i\chi_\g^{-1/3})n^{-1/3}$ as $n\to\infty$.
Multiplying \eqref{Phioneterm} for all $i$ and for $W$ and $Z$ results in
\begin{equation}\label{Phiratio}
\frac{\Phi^\g(W)}{\Phi^\g(Z)}\sim\prod_{i=1}^m\frac{-\log q(z-b_i\chi_\g^{-1/3})n^{-1/3}(q;q)_\infty}{-\log q(w-b_i\chi_\g^{-1/3})n^{-1/3}(q;q)_\infty}
=\prod_{i=1}^m\frac{z-b_i\chi_\g^{-1/3}}{w-b_i\chi_\g^{-1/3}}
\end{equation}
where the ratio $\Phi^\g(W)/\Phi^\g(Z)$ remains bounded on their contours $|wn^{-1/3}-\theta|<\delta$ and $|zn^{-1/3}-\theta|<\delta$.
The right-hand side of \eqref{Phiratio} after the rescaling of the variables $w$ and $z$ by $\chi_\g^{1/3}$ simplifies exactly
to the extra factor in the kernel $K_{\rm BBP,\bf b}$ of the Baik--Ben Arous--P\'ech\'e distribution in \eqref{defKBBP} compared to the Airy kernel.
Since there is an $e^{-c(w^3+z^3)}$ decay in the $w$ and $z$ variables along their respective integration contours,
it remains integrable also with the extra factor in \eqref{Phiratio}.
Therefore, the rest of the proof of Proposition~\ref{prop:gconvergence} applies for this case
and it yields the convergence \eqref{gBBP} for the geometric $q$-TASEP.

\subsection{Finite GUE limit}
\label{ss:Gaussian}

The main step in the proof of the finite GUE limit in the third part of Theorem~\ref{thm:gmain} is the analogue of Proposition~\ref{prop:gsteep}
about the steep decent property of the integration contours.
We first set
\begin{equation}\label{zchoiceg2}
\wt z^\g_x=-q^{-(g_\g+1)n-\frac{\sigma_\g^{1/2}}{\log q}xn^{1/2}}
\end{equation}
and we define
\begin{align}
g_0^\g(W)&=\kappa_\g\log(aq^W;q)_\infty-\log(q^W/a;q)_\infty+(g_\g+1)\log q\,W,\label{defg0g}\\
g_1^\g(W)&=\sigma_\g^{1/2}xW.\label{defg1g}
\end{align}
Then the rescaled kernel can be written as
\begin{multline}
q^W\log qK^\g_{\wt z^\g_x}(q^W,q^{W'})\\
=\frac{q^W\log q}{2\pi i}\int_{\log_q\ami+i\R}\frac{\d Z}{q^Z-q^{W'}}\frac\pi{\sin(\pi(Z-W))}
\frac{e^{ng_0^\g(W)+n^{1/2}g_1^\g(W)}}{e^{ng_0^\g(Z)+n^{1/2}g_1^\g(Z)}}\frac{\Phi^\g(W)}{\Phi^\g(Z)}.
\end{multline}
The function $g_0^\g$ has a critical point at $\log_q\ami$ and its Taylor expansion around the critical point is
\begin{equation}\label{g0Taylorg}
g_0^\g(W)=g_0^\g(\log_q\ami)-\frac{\sigma_\g}2(W-\log_q\ami)^2-\frac{\xi_\g}3(W-\log_q\ami)^3+\O((W-\log_q\ami)^4)
\end{equation}
where
\begin{equation}
\xi_\g=\frac12\left(\kappa_\g\Psi_q''(\log_q\ami+\log_qa)-\Psi_q''(\log_q\ami-\log_qa)\right).
\end{equation}

For $\delta>0$ small enough, we define a new integration contour
\begin{equation}\label{defCdelta}
\C_\theta^{\log_q\ami,\delta}=V_{\log_q\ami,19\pi/24}^\delta\cup\left\{W(s)=\log_q(a-(a-q^\theta)e^{is}),s\in(-\pi+\delta',\pi-\delta']\right\}
\end{equation}
for the variables $W,W'$ where $\delta'>0$ is chosen in a way that the endpoints of the two parts of the contour above coincide.
We do not give the explicit connection between $\delta$ and $\delta'$,
but we remark that they have the property that $\delta'\to0$ as $\delta\to0$.
The circular part of the contour $\C_\theta^{\log_q\ami,\delta}$ coincides with $\C_\theta$.
The angle $19\pi/24$ is chosen to be in $(3\pi/4,5\pi/6)$.
The contour $\C_\theta^{\log_q\ami,\delta}$ is shown in Figure~\ref{fig:contour}.

\begin{figure}
\centering
\psfrag{C}{$\C_\theta^{\log_q\ami,\delta}$}
\psfrag{Ct}{$\wt\C_\theta^{\,\log_q\ama,\delta}$}
\psfrag{al}{$\log_q\ami$}
\psfrag{alt}{$\log_q\ama$}
\psfrag{th}{$\theta$}
\psfrag{a}{$\log_qa$}
\includegraphics[width=150pt]{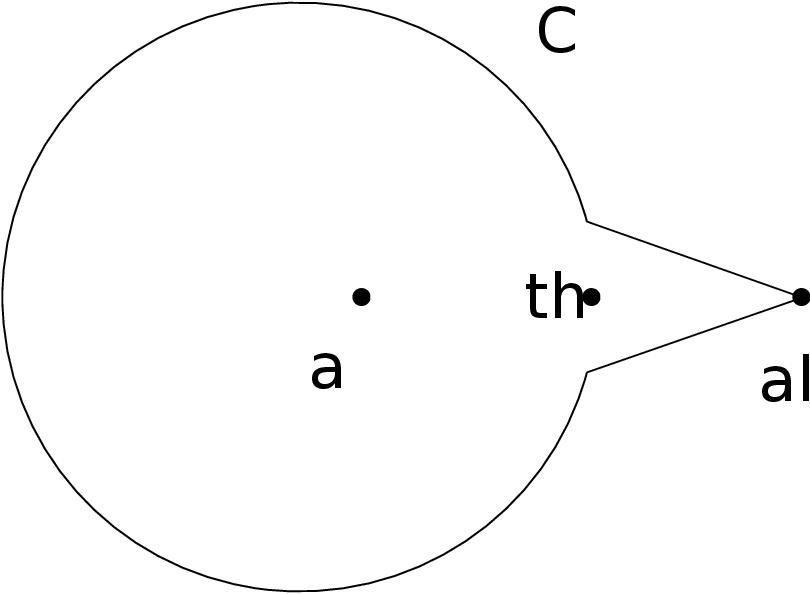}\hspace*{5em}
\includegraphics[width=150pt]{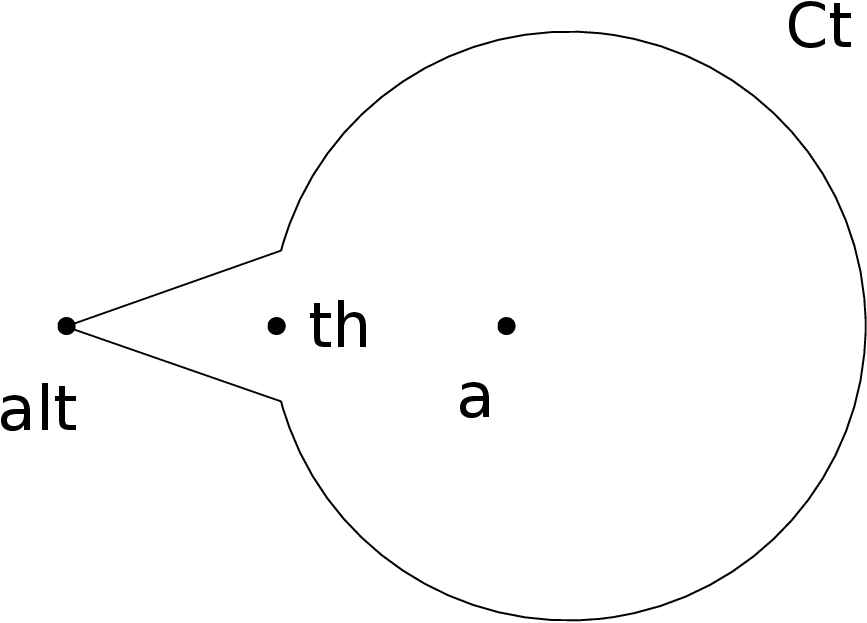}
\caption{Left: the integration contour $\C_\theta^{\log_q\ami,\delta}$ used for proving the finite GUE limit of the rescaled particle position
in geometric $q$-TASEP.
Right: the contour $\wt\C_\theta^{\,\log_q\ama,\delta}$ for the finite GUE limit in geometric $q$-PushTASEP.
\label{fig:contour}}
\end{figure}

\begin{proposition}\label{prop:gsteep2}
Assume that the conditions of Theorem~\ref{thm:gmain} hold.
\begin{enumerate}
\item For given $q,a,\ami$ satisfying the assumptions of Theorem~\ref{thm:gmain} there is an $\varepsilon>0$ with the following property.
For all $\theta\in(\log_q\ami-\varepsilon,\log_q\ami)$ the contour $\C_\theta^{\log_q\ami,\delta}$ is of steep descent for the function $\Re(g_0^\g)$
in the sense that the function attains its maximum at $\ami$ and it decreases away from this point.
\item The function $-\Re(g_0^\g)$ is periodic along $\D_{\log_q\ami}$ with period length $2\pi|\log q|$.
The period symmetric about the real axis of the contour $\D_{\log_q\ami}$ which corresponds to $t\in(-\pi|\log q|,\pi|\log q|]$
is of steep descent for the function $-\Re(g_0^\g)$ in the following sense.
The function attains its maximum at $\ami$ which corresponds to $t=0$, it increases for $t\in(-\pi|\log q|,0)$ and it decreases for $t\in(0,\pi|\log q|)$.
\end{enumerate}
\end{proposition}

Proposition~\ref{prop:gsteep2} is proved in Subsection~\ref{ss:gsteep}.
Now we are ready to prove the third part of Theorem~\ref{thm:gmain} for the geometric $q$-TASEP.
Similarly to Lemma~\ref{lemma:gLHS}, the left-hand side of \eqref{gfinite} has the same limit as
\begin{equation}
\lim_{n\to\infty}\E\left(\frac1{(\wt z^\g_xq^{x_n^\g(\kappa_\g n)+n};q)_\infty}\right)=\lim_{n\to\infty}\P(\eta_n^\g<x).
\end{equation}
The convergence of the right-hand side of \eqref{gfinite} proceeds via the same steps as in the proof of the Tracy--Widom limit.
Furthermore, a similar convergence result for $q$-TASEP with finite GUE limit is proved in Subsection 2.3 of~\cite{Bar14}.
Hence we only indicate the main steps in the present proof.

We deform the integration paths in \eqref{gfinite} to the steep descent contours $W\in\C_\theta^{\log_q\ami,\delta}$ and $Z\in\D_{\log_q\ami}$
in the same way as in Step 1 in the proof of Proposition~\ref{prop:glocalization}.
The assumptions of Theorem~\ref{thm:gmain} ensure that no poles are crossed during the deformation.
More precisely, there is a pole at $W=\log_q\ami$ which has to be encircled by the $W$ contour,
hence the $W$ contour should locally be modified on the scale $n^{-1/2}$ in the neighbourhood of $\log_q\ami$ so that it avoids this pole
and it crosses the real axis to the right of the pole.
The poles at $W=\log_qa_i-1,\log_qa_i-2,\dots$ all remain outside of $\C_\theta^{\log_q\ami,\delta}$ if the condition
\begin{equation}\label{galphacond2}
\log_q(2a-q^\theta)>\log_q\ami-1
\end{equation}
holds since $\log_qa_i\in[\log_qa,\log_q\ami]$.
The left-hand side of \eqref{galphacond2} is the point where the contour $\C_\theta^{\log_q\ami,\delta}$ crosses the real axis
which can be lower bounded by $\log_q(2a-\ami)$.
Hence we do not cross any pole if $\log_q(2a-\ami)>\log_q\ami-1$ which is condition \eqref{galphacond}.

The $Z$ contour cannot cross the $W$ contour, furthermore the condition $\Re(Z-W)>0$ should hold along the chosen contours
in order to use the formula
\begin{equation}\label{integraltrick}
\frac1{Z-W}=\int_{\R_+}\d\lambda e^{-\lambda(Z-W)}
\end{equation}
later in this proof.
The final choice for the $Z$ contour can be $\D_{\log_q\ami+n^{-1/2}}$.
This contour does not necessarily have the steep descent property for $-\Re(g_0^\g)$ but it is still true that along the contour the function $-\Re(g_0^\g)$
is maximal at the intersection with the real line for $n$ large enough which can be seen as follows.
Locally around $\log_q\ami$, the contour $\D_{\log_q\ami+n^{-1/2}}$ is of steep descent by the Taylor expansion \eqref{g0Taylorg}.
By the continuity of the function $g_0^\g$ away from its poles,
the values of $-\Re(g_0^\g)$ along the rest of $\D_{\log_q\ami+n^{-1/2}}$ are smaller than in a neighbourhood of $\log_q\ami$ for $n$ large enough.

Next we can localize the integrals to the neighbourhood of $\log_q\ami$ similarly to Step 2 in the proof of Proposition~\ref{prop:glocalization}.
Step 3 in the proof of Proposition~\ref{prop:glocalization} is not necessary here,
we can work with the vertical short contour $Z\in V_{\log_q\ami+n^{-1/2},\pi/2}$.
The overall error caused by the localization is of order $\O(\exp(-c\delta^2n))$.
This proves an analogue of Proposition~\ref{prop:glocalization}.

What remains to show is that the Fredholm determinant on the short contours converges to $G_k(x)$ on the right-hand side of \eqref{gGaussian}.
This is done in the following steps.
First, we apply the change of variables $W=\log_q\ami+wn^{-1/2}$, $W'=\log_q\ami+w'n^{-1/2}$, $Z=\log_q\ami+zn^{-1/2}$
and use Taylor approximation of $g_0^\g$ and $g_1^\g$ around $\log_q\ami$.
Similarly to the BBP regime, the $\Phi^\g(W)/\Phi^\g(Z)$ factor remains bounded when the distance of $Z$ from the pole at $\log_q\ami$ is at most a constant.
By a calculation similar to \eqref{Phioneterm}--\eqref{Phiratio}, the behaviour around $\log_q\ami$ is
\begin{equation}\label{PhiGaussian}
\frac{\Phi^\g(W)}{\Phi^\g(Z)}\sim\left(\frac zw\right)^k
\end{equation}
where $k$ is the multiplicity of $\ami$ within $\{a_1,a_2,\dots\}$.
Hence by using the integral representation \eqref{integraltrick} in the rescaled variables for the factor coming from the sine function in the denominator,
we obtain that the Fredholm determinant on the short contours differs by at most $\O(n^{-1/2})$ from $\det(\id-AB)$ where
\begin{equation}
A(w,\lambda)=e^{-\sigma_\g w^2/2+(\sigma_\g^{1/2}x+\lambda)w}\frac1{w^k},\qquad
B(\lambda,w)=\frac1{2\pi i}\int\frac{\d z}{z-w}\,e^{\sigma_\g z^2/2-(\sigma_\g^{1/2}x+\lambda)z}z^k.
\end{equation}
Here $\lambda\in\R_+$ and the previous short contours are blown up by $n^{1/2}$ to get the contours for $w$ and $z$,
that is the contour for $w$ is the local modification of $V_{0,19\pi/24}^{\delta n^{1/2}}$ that for $z$ is the modification of $V_{1,\pi/2}^{\delta n^{1/2}}$.
After the change of variables, the $w$ and $z$ contours remain at least a constant distance apart from each other and from the pole at $0$
and they cross the real axis to the right of $0$.
Finally the angle of the contours can be adjusted to coincide with the choice in \eqref{defHk}.
After rescaling by $\sigma_\g^{1/2}$, we get that $\det(\id-BA)=\det(\id-H_k)_{L^2((x,\infty))}$ proving \eqref{gGaussian} for the geometric $q$-TASEP.

\section{Asymptotics for the geometric $q$-PushTASEP}
\label{s:p}

The following $q$-Laplace transform formula holds for the particle positions in the geometric $q$-PushTASEP
with particle specific rates $a_i$ and time dependent parameters $\alpha_j$.

\begin{theorem}\label{thm:pfinite}
For any $z\in\mathbb C\setminus\R_+$,
\begin{equation}\label{pfinite}
\E\left(\frac1{(zq^{x_n^\p(t)+n};q)_\infty}\right)=\det\left(\id+K_z^\p\right)_{L^2(C_{a_1^{-1},\dots,a_n^{-1}})}
\end{equation}
where
\begin{equation}
K_z^\p(w,w')=\frac1{2\pi i}\int_{1/2+i\R}\frac\pi{\sin(-\pi s)}(-z)^s\frac{h(q^sw)}{h(w)}\frac1{q^sw-w'}\,\d s
\end{equation}
with
\begin{equation}
h(w)=\frac{\prod_{i=1}^n(a_iw;q)_\infty}{\prod_{j=1}^t(\alpha_j/w;q)_\infty}.
\end{equation}
The integration contour $C_{a_1^{-1},\dots,a_n^{-1}}$ in \eqref{pfinite} is positively oriented and it contains the poles at $a_1^{-1},\dots,a_n^{-1}$
and no other singularities of the kernel.
\end{theorem}

\begin{proof}
The equality in distribution
\begin{equation}
x_n^\p(t)\stackrel{\d}=-\lambda_1^{(n)}(t)-n
\end{equation}
holds for all $n\ge1$ and $t\ge0$ where $\lambda_1^{(n)}(t)$ is a marginal of the $q$-Whittaker 2d-growth model
under the $q$-Whittaker measure with a pure alpha specialization.
By subsection 7.4 in~\cite{MP17}, Theorem 3.3 in~\cite{BCFV14} provides an infinite contour Fredholm determinant formula
for the $q$-Laplace transform of the particle positions in the geometric $q$-PushTASEP.
Theorem~\ref{thm:pfinite} states the corresponding finite contour formula which is proved implicitly in~\cite{BCFV14} as follows.
Theorem 3.3 in~\cite{BCFV14} is deduced from the $q$-moment formulas of $\lambda_1^{(n)}$ given in Proposition 3.6 of~\cite{BCFV14}.
It follows from the proof of the proposition that the integration contours in these moment formulas can be replaced by small circles around the poles at $a_i^{-1}$,
hence the final Fredholm expression in Theorem 3.3 in~\cite{BCFV14} can also be written using finite contours.
We omit further details.
We mention that in the homogeneous case the finite contour formula for the geometric $q$-PushTASEP can be found explicitly in the literature.
It is the $\nu=0$ special case of the one in Conjecture 3.11 in~\cite{CMP19} about the more general $q$-Hahn PushTASEP, see the discussion before the statement of the conjecture.
The parameter $\mu$ in~\cite{CMP19} corresponds to $\alpha a$ and \eqref{pfinite} can be obtained by replacing the variable $w$ by $w/a$.
\end{proof}

The proof of Theorem~\ref{thm:pmain} is now formally very similar to that of Theorem~\ref{thm:gmain}, hence we only indicate the differences below.
First, the analogue of Lemma~\ref{lemma:gLHS} can be proven in the same way.
\begin{lemma}\label{lemma:pLHS}
Let
\begin{equation}\label{zchoicep}
z^\p_x=-q^{(f_\p-1)n-\frac{\chi_\p^{1/3}}{\log q}xn^{1/3}}.
\end{equation}
Then for the left-hand side of \eqref{pfinite} on the time scale $t=\kappa_\p n$
\begin{equation}\label{qLaplacep}
\lim_{n\to\infty}\E\left(\frac1{(z^\p_xq^{x_n^\p(\kappa_\p n)+n};q)_\infty}\right)=\lim_{n\to\infty}\P(\xi_n^\p<x).
\end{equation}
\end{lemma}

After substituting $\alpha_1=\alpha_2=\dots=a$ and $a_{m+1}=a_{m+2}=\dots=a$,
we perform the change of variables $w=q^{-W}, w'=q^{-W'}, s=W-Z$ in the Fredholm determinant on the right-hand side of \eqref{pfinite}.
This Fredholm determinant for the choice \eqref{zchoicep} is equal to the Fredholm determinant of the rescaled kernel
\begin{multline}\label{Krescaledp}
-q^{-W}\log qK^\p_{z^\p_x}(q^{-W},q^{-W'})\\
=-\frac{q^{-W}\log q}{2\pi i}\int_{\theta+i\R}\frac{\d Z}{q^{-Z}-q^{-W'}}\frac\pi{\sin(\pi(W-Z))}
\frac{e^{nf_0^\p(W)+n^{1/3}f_2^\p(W)}}{e^{nf_0^\p(Z)+n^{1/3}f_2^\p(Z)}}\frac{\Phi^\p(W)}{\Phi^\p(Z)}
\end{multline}
where
\begin{align}
f_0^\p(W)&=\kappa_\p\log(aq^W;q)_\infty-\log(aq^{-W};q)_\infty+(f_\p-1)\log q\,W,\label{deff0p}\\
f_2^\p(W)&=-\chi_\p^{1/3}xW
\end{align}
and
\begin{equation}\label{defPhip}
\Phi^\p(W)=\frac{(aq^{-W};q)_\infty^m}{\prod_{i=1}^m(a_iq^{-W};q)_\infty}.
\end{equation}
The integration contour for the kernel in \eqref{Krescaledp} for $W$ has to encircle $\log_qa$ and $\log_qa_i$ for all $i=1,\dots,m$
but no other singularities of the kernel.
We define the integration contour
\begin{equation}\label{defCtilde}
\wt\C_\theta=\left\{W(s)=\log_q(a+(q^\theta-a)e^{is}),s\in(-\pi,\pi]\right\}.
\end{equation}
The contour $\wt\C_\theta$ is the image of the circle around $a$ that passes through $q^\theta$ under the map $\log_q$.
We will choose the contour for $W$ to be $\wt\C_\theta$ and the contour for $Z$ to be $\D_\theta$ defined in \eqref{defCD} which is a vertical line.
We verify in the proof of Proposition~\ref{prop:plocalization} that no singularities of the integrand are crossed during the deformation.

Taylor expansion around $\theta$ yields that
\begin{equation}\label{f0pTaylor}
f_0^\p(W)=f_0^\p(\theta)+\frac{\chi_\p}3(W-\theta)^3+\O((W-\theta)^4).
\end{equation}
The following steep descent properties for the contours $\wt\C_\theta$ and $\D_\theta$ are the key for the asymptotics of geometric $q$-PushTASEP.

\begin{proposition}\label{prop:psteep}
Assume that the conditions of Theorem~\ref{thm:pmain} hold.
\begin{enumerate}
\item The contour $\wt\C_\theta$ is of steep descent for the function $\Re(f_0^\p)$ in the sense that the function attains its maximum $q^\theta$
which corresponds to $s=0$, it increases for $s\in(-\pi,0)$ and it decreases for $s\in(0,\pi)$.
\item The function $-\Re(f_0^\p)$ is periodic along $\D_\theta$ with period length $2\pi|\log q|$.
The period symmetric about the real axis of the contour $\D_\theta$ which corresponds to $t\in(-\pi|\log q|,\pi|\log q|]$ is of steep descent for the function $-\Re(f_0^\p)$ in the following sense.
The function attains its maximum at $q^\theta$ which corresponds to $t=0$, it increases for $t\in(-\pi|\log q|,0)$ and it decreases for $t\in(0,\pi|\log q|)$.
\end{enumerate}
\end{proposition}

We define
\begin{equation}
K^\p_{x,\delta}(W,W')
=\frac{q^{-W}\log q}{2\pi i}\int_{V_{\theta,\pi-\varphi}^\delta}\frac{\d Z}{q^{-Z}-q^{-W'}}\frac\pi{\sin(\pi(W-Z))}
\frac{e^{nf_0^\p(W)+n^{1/3}f_2^\p(W)}}{e^{nf_0^\p(Z)+n^{1/3}f_2^\p(Z)}}\frac{\Phi^\g(Z)}{\Phi^\g(W)}
\end{equation}
where $W,W'\in V_{\theta,\varphi}^\delta$ with the contour defined in \eqref{defVcont}.

\begin{proposition}\label{prop:plocalization}
For any fixed $\varepsilon>0$ small enough, there are $\delta>0$, $\varphi\in(\pi/6,\pi/2)$ and $n_0$ such that for all $n>n_0$
\begin{equation}
\left|\det(\id+K^\p_{z_x})_{L^2(C_{a_1^{-1},\dots,a_n^{-1}})}-\det(\id-K^\p_{x,\delta})_{L^2(V_{\theta,\varphi}^\delta)}\right|<\varepsilon.
\end{equation}
\end{proposition}

\begin{proof}[Proof of Proposition~\ref{prop:plocalization}]
The proof is completely analogous to that of Proposition~\ref{prop:glocalization}.
We give the detailed description of the first step below about the contour deformation.

We deform the integration contours in \eqref{Krescaledp} to $W\in\wt\C_\theta$ and $Z\in\D_\theta$ without crossing any singularity of the integrand.
Along the new contours $W\in\wt\C_\theta$ and $Z\in\D_\theta$, we have that $\Re(W)\in[\theta,\log_q(2a-q^\theta)]$ and $\Re(Z)=\theta$.
Hence in the denominator $\sin(\pi(W-Z))\neq0$ if $\theta+1>\log_q(2a-q^\theta)$ which is equivalent to \eqref{pcondpole}.
The poles at $Z=-\log_qa,-\log_qa-1,-\log_qa-2,\dots$ coming from one of the $q$-Pochhammer symbols in the denominator are negative
hence they are certainly not crossed.
The other $q$-Pochhammer symbol results in the poles at $W=\log_qa_i,\log_qa_i+1,\log_qa_i+2,\dots$ for $i=1,\dots,m$.
The poles at $\log_qa_i$ are all contained in the interval $[\log_q\ama,\log_qa]\subset[\theta,\log_qa]$
by assumption in the first case of Theorem~\ref{thm:pmain} which are encircled.
To guarantee that the remaining ones at $\log_qa_i+1,\log_qa_i+2,\dots$ are avoided the condition $\theta+1>\log_q(2a-q^\theta)$ is enough
which is the same as above and it is equivalent to \eqref{pcondpole}.

The remaining two steps of the proof about the localization to the short contours and about the deformation of short contours are identical.
The steep descent properties of Proposition~\ref{prop:psteep} justify the localization step.
The only difference in the rest of the proof compared to that of Proposition~\ref{prop:glocalization} is the replacement of $\varphi$ by $\pi-\varphi$.
That is, the short contour for $W$ is going to be $V_{0,\varphi}^\delta$ and that for $Z$ is chosen to be $V_{0,\pi-\varphi}^\delta$.
Further details are omitted here.
\end{proof}

Finally, the convergence $\det(\id-K^\p_{x,\delta})_{L^2(V_{\theta,\varphi}^\delta)}\to F_{\rm GUE}(x)$ follows
exactly in the same way as in the proof of Proposition~\ref{prop:gconvergence}.
The only difference is the sign of the cubic term in the Taylor expansion \eqref{f0gTaylor} of $f_0^\g$ compared to \eqref{f0pTaylor} that of $f_0^\p$.
The sign difference is the reason for the opposite choice of the integration contours for the variables $W$ and $Z$.
This proves the Tracy--Widom limit for the geometric $q$-PushTASEP.
The proof of the Baik--Ben Arous--P\'ech\'e limit follows exactly the same step as for the geometric $q$-TASEP.

For the finite GUE limit, we first set
\begin{equation}\label{zchoicep2}
\wt z^\p_x=-q^{(g_\p-1)n-\frac{\sigma_\p^{1/2}}{\log q}xn^{1/2}}
\end{equation}
and we define
\begin{align}
g_0^\p(W)&=\kappa_\p\log(aq^W;q)_\infty-\log(aq^{-W};q)_\infty+(g_\p-1)\log q\,W,\label{defg0p}\\
g_1^\p(W)&=-\sigma_\p^{1/2}xW.\label{defg1p}
\end{align}
Then the rescaled kernel can be written as
\begin{multline}
-q^{-W}\log qK^\p_{\wt z^\p_x}(q^{-W},q^{-W'})\\
=-\frac{q^{-W}\log q}{2\pi i}\int_{\log_q\ama+i\R}\frac{\d Z}{q^{-Z}-q^{-W'}}\frac\pi{\sin(\pi(W-Z))}
\frac{e^{ng_0^\p(W)+n^{1/2}g_1^\p(W)}}{e^{ng_0^\p(Z)+n^{1/2}g_1^\p(Z)}}\frac{\Phi^\p(W)}{\Phi^\p(Z)}.
\end{multline}
The function $g_0^\p$ has a critical point at $\log_q\ama$ and its Taylor expansion around the critical point is
\begin{equation}\label{g0Taylorp}
g_0^\p(W)=g_0^\p(\log_q\ama)-\frac{\sigma_\p}2(W-\log_q\ama)^2+\frac{\xi_\p}3(W-\log_q\ama)^3+\O((W-\log_q\ama)^4)
\end{equation}
where
\begin{equation}
\xi_\p=-\frac12\left(\kappa_\p\Psi_q''(\log_qa+\log_q\ama)+\Psi_q''(\log_qa-\log_q\ama)\right).
\end{equation}

For $\delta>0$ small enough, we define the contour
\begin{equation}\label{defCdeltatilde}
\wt\C_\theta^{\,\log_q\ama,\delta}=V_{\log_q\ama,5\pi/24}^\delta\cup\left\{W(s)=\log_q(a+(q^\theta-a)e^{is}),s\in(-\pi+\delta',\pi-\delta']\right\}
\end{equation}
for the variables $W,W'$ where $\delta'>0$ is such that the endpoints of the two parts of the contour above coincide.
The angle $5\pi/24$ is chosen to be in $(\pi/6,\pi/4)$.
The circular part of the contour $\wt\C_\theta^{\,\log_q\ama,\delta}$ coincides with $\wt\C_\theta$, see Figure~\ref{fig:contour}.

\begin{proposition}\label{prop:psteep2}
Assume that the conditions of Theorem~\ref{thm:pmain} hold.
\begin{enumerate}
\item For given $q,a,\ama$ satisfying the assumptions of Theorem~\ref{thm:pmain} there is an $\varepsilon>0$ with the following property.
For all $\theta\in(\log_q\ama,\log_q\ama+\varepsilon)$ the contour $\wt\C_\theta^{\,\log_q\ama,\delta}$ is of steep descent
for the function $\Re(g_0^\p)$ in the sense that the function attains its maximum at $\ama$ and it increases away from this point.
\item The function $-\Re(g_0^\p)$ is periodic along $\D_{\log_q\ama}$ with period length $2\pi|\log q|$.
The period symmetric about the real axis of the contour $\D_{\log_q\ama}$ which corresponds to $t\in(-\pi|\log q|,\pi|\log q|]$
is of steep descent for the function $-\Re(g_0^\p)$ in the following sense.
The function attains its maximum at $\ama$ which corresponds to $t=0$, it increases for $t\in(-\pi|\log q|,0)$ and it decreases for $t\in(0,\pi|\log q|)$.
\end{enumerate}
\end{proposition}

Again we have the convergence of the left-hand side of \eqref{pfinite}
\begin{equation}
\lim_{n\to\infty}\E\left(\frac1{(\wt z^\p_xq^{x_n^\p(\kappa_\p n)+n};q)_\infty}\right)=\lim_{n\to\infty}\P(\eta_n^\p<x).
\end{equation}
For the convergence of the right-hand side, we deform the integration paths in \eqref{pfinite} as follows.
For the variable $W$, we choose the steep descent contour $\wt\C_\theta^{\,\log_q\ama,\delta}$
locally modified in the $n^{-1/2}$ neighbourhood of the pole at $W=\log_q\ama$ in a way that it crosses the real axis to the left of the pole.
The poles at $W=\log_qa_i+1,\log_qa_i+2,\dots$ remain outside of $\wt\C_\theta^{\,\log_q\ama,\delta}$ if
\begin{equation}\label{palphacond2}
\log_q(2a-q^\theta)<\log_q\ama+1
\end{equation}
holds because $\log_qa_i\in[\log_q\ama,\log_qa]$.
The contour $\wt\C_\theta^{\,\log_q\ama,\delta}$ crosses the real axis at $\log_q(2a-q^\theta)$ which is upper bounded by $\log_q(2a-\ama)$.
Hence \eqref{palphacond2} holds if $\log_q(2a-\ama)<\log_q\ama+1$ which is equivalent to \eqref{palphacond}.

The contour for $Z$ is chosen to be $\D_{\log_q\ama-n^{-1/2}}$ so that it does not intersect the $W$ contour
and the condition $\Re(W-Z)>0$ remains valid along the contours.
Similarly to the finite GUE limit for the geometric $q$-TASEP, the contour $\D_{\log_q\ama-n^{-1/2}}$ only locally enjoys the steep descent property,
but the value of the function $-\Re(g_0^\p)$ along of the contour is smaller than in an $n^{-1/2}$ neighbourhood of $\log_q\ama$ if $n$ is large enough.

The rest of the argument for the proof of the third part of Theorem~\ref{thm:pmain} is the same as for the geometric $q$-TASEP.
In particular by using \eqref{integraltrick} with the role of $W$ and $Z$ exchanged,
we get that the Fredholm determinant with the localized contours differs by at most $\O(n^{-1/2})$ from $\det(\id-AB)$ where
\begin{equation}
A(w,\lambda)=e^{-\sigma_\p w^2/2-(\sigma_\p^{1/2}x+\lambda)w}\frac1{w^k},\qquad
B(\lambda,w)=\frac1{2\pi i}\int\frac{\d z}{z-w}\,e^{\sigma_\p z^2/2+(\sigma_\p^{1/2}x+\lambda)z}z^k
\end{equation}
with $\lambda\in\R_+$ and where the contour for $w$ is the local modification of $V_{0,5\pi/24}^{\delta n^{1/2}}$
and that for $z$ is the modification of $V_{-1,\pi/2}^{\delta n^{1/2}}$.
Finally the contours in \eqref{defHk} for the kernel $H_k$ can be obtained by a change of variables $w\to-w$ and $z\to-z$.

\section{Asymptotics for the $q$-PushASEP}
\label{s:a}

We prove Theorem~\ref{thm:amain} in this section.
The following Fredholm determinant formula for the $q$-Laplace transform of the particle position in $q$-PushASEP
first appears in~\cite{CP15} as Conjecture 1.4 and it is proved as Theorem 7.10 in~\cite{MP17}.

\begin{theorem}\label{thm:afinite}
For any $z\in\mathbb C\setminus\R_+$,
\begin{equation}\label{afinite}
\E\left(\frac1{(zq^{x_n^\a(t)+n};q)_\infty}\right)=\det\left(\id+K_z^\a\right)_{L^2(C_1)}
\end{equation}
where
\begin{equation}
K_z^\a(w,w')=\frac1{2\pi i}\int_{1/2+i\R}\frac\pi{\sin(-\pi s)}(-z)^s\frac{h(q^sw)}{h(w)}\frac1{q^sw-w'}\,\d s
\end{equation}
with
\begin{equation}
h(w)=(w;q)_\infty^n e^{tRw+tLw^{-1}}.
\end{equation}
The integration contour $C_1$ is positively oriented small circle around $1$ which does not contain any other singularity of the kernel.
\end{theorem}

Let us define the integration contour
\begin{equation}
\ol\C_\theta=\left\{W(s)=\log_q(1-(1-q^\theta)e^{is}),s\in(-\pi,\pi]\right\}.
\end{equation}
Then we choose
\begin{equation}\label{zchoicea}
z^\a_x=-q^{-(f_\a+1)n-\frac{\chi_\a^{1/3}}{\log q}xn^{1/3}}
\end{equation}
for which
\begin{equation}
\lim_{n\to\infty}\E\left(\frac1{(z^\a_xq^{x_n^\a(\kappa_\a n)+n};q)_\infty}\right)=\lim_{n\to\infty}\P(\xi_n^\a<x)
\end{equation}
holds using that $\chi_\a>0$ by Lemma~\ref{lemma:psimonotonicity}.
Then we apply the change of variables $w=q^W$, $w'=q^{W'}$, $s=Z-W$ to get the rescaled kernel
\begin{equation}
q^W\log qK^\a_{z^\a_x}(q^W,q^{W'})
=\frac{q^W\log q}{2\pi i}\int_{\theta+i\R}\frac{\d Z}{q^Z-q^{W'}}\frac\pi{\sin(\pi(Z-W))}
\frac{e^{nf_0^\a(Z)+n^{1/3}f_2^\a(Z)}}{e^{nf_0^\a(W)+n^{1/3}f_2^\a(W)}}
\end{equation}
where
\begin{align}
f_0^\a(W)&=\log(q^W;q)_\infty+R\kappa_\a q^W+L\kappa_\a q^{-\theta}-(f_\a+1)\log q\,W,\\
f_2^\a(W)&=-\chi_\a^{1/3}xW.
\end{align}
The integration contour for $W$ has to encircle $0$ and no other singularity,
but by the fact that $\theta$ is chosen to be small one can avoid crossing any pole during the deformation of the contour to $\ol\C_\theta$.
For the variable $Z$, the poles coming from the sine in the denominator are not approached if we deform its contour to $\D_\theta$.
By Taylor expansion we have that
\begin{equation}\label{f0aTaylor}
f_0^\a(W)=f_0^\a(\theta)+\frac{\chi_\a}3(W-\theta)^3+\O((W-\theta)^4)
\end{equation}
as $W\to\theta$.
The crucial part of the asymptotics is the following steep descent property along the chosen contours.

\begin{proposition}\label{prop:asteep}
Let $q\in(0,1)$ and $R,L\ge0$ be chosen as in Theorem~\ref{thm:amain}.
\begin{enumerate}
\item
There is a $\theta^*>0$ such that for all $\theta\in(0,\theta^*)$ the contour $\ol\C_\theta$ is of steep descent for the function $-\Re(f_0^\a)$
in the sense that the function attains its maximum $q^\theta$ which corresponds to $s=0$, it increases for $s\in(-\pi,0)$ and it decreases for $s\in(0,\pi)$.
\item
The function $\Re(f_0^\a)$ is periodic along $\D_\theta$ with period $2\pi|\log q|$.
The period symmetric about the real axis which corresponds to $t\in[-\pi|\log q|,\pi|\log q|]$ is of steep descent
for the function $-\Re(f_0^\g)$ in the following sense.
The function attains its maximum at $q^\theta$ which corresponds to $t=0$, it increases for $t\in(-\pi|\log q|,0)$ and it decreases for $t\in(0,\pi|\log q|)$.
\end{enumerate}
\end{proposition}

Then the asymptotic analysis follows exactly the same steps as in the case of geometric $q$-TASEP:
localization of the integration to the neighbourhood of the double critical point of $f_0^\a$ at $\theta$,
Taylor expansion and reformulation of the kernel using the identity $\det(\id-AB)=\det(\id-BA)$.
We omit further details here and conclude Theorem~\ref{thm:amain}.

\section{Steep descent contours}
\label{s:steep}

This section contains the proofs of all steep descent property statements in this paper.
We define the function
\begin{equation}\label{defg}
g(b,s)=\frac{b\sin s}{1+b^2-2b\cos s}
\end{equation}
as in~\cite{Vet15} which will appear frequently in derivatives below.
The function has the $g(1/b,s)=g(b,s)$ symmetry in the first variable.
It has a certain monotonicity property in the first variable.
To state it, let
\begin{equation}\label{defh}
h(b,s)=\frac{(1-b)^2}bg(b,s)=\frac{(1-b)^2\sin s}{1+b^2-2b\cos s}
\end{equation}
be the rescaled version of $g(b,s)$ by its derivative $\frac{\d}{\d s}g(b,s)|_{s=0}=\frac b{(1-b)^2}$.
The right-hand side of \eqref{defh} is defined for $b=0$ as well.
Note that the $b\leftrightarrow1/b$ symmetry in the first variable is inherited to $h$, that is, $h(1/b,s)=h(b,s)$.

\begin{lemma}[Lemma 6.5 in~\cite{Vet15}]\label{lemma:hcompare}
The function $b\mapsto h(b,s)$ is non-increasing for $b\in[-1,1]$ and $0\le s\le\pi$,
that is, if $-1\le b\le c\le1$ and $0\le s\le\pi$, then $h(b,s)\ge h(c,s)$.
\end{lemma}

The function $h(b,s)$ is defined in a way to give sharp bounds around $s=0$ on the derivative of the function $\Re(f_0^\g)$
which appear in the geometric $q$-TASEP.
To show the required steep descent property in Proposition~\ref{prop:gsteep}, the use of $h(b,s)$ suffices.
In the case of the geometric $q$-PushTASEP these bounds are not enough to prove Proposition~\ref{prop:psteep} about $\Re(f_0^\p)$,
hence we introduce and use the following function which describes the behaviour of the derivative around $s=\pi$.
Let
\begin{equation}\label{defe}
e(b,s)=\frac{(1+b)^2}{1+b^2-2b\cos s}
\end{equation}
with which one can write
\begin{equation}\label{hwithe}
h(b,s)=\left(1-\frac{2b(1-\cos s)}{(1+b)^2}e(b,s)\right)\sin s.
\end{equation}

\begin{lemma}\label{lemma:eproperties}
The function $b\mapsto e(b,s)$ is non-decreasing for $b\in[-1,1]$.
Further for all $s\in\R$, $((1+b)/(1-b))^2\le e(b,s)\le1$ holds for $b\le0$ and $1\le e(b,s)\le((1+b)/(1-b))^2$ for $b\ge0$.
\end{lemma}

\begin{proof}
The monotonicity follows by computing the derivative
\begin{equation}
\frac{\d}{\d b}e(b,s)=\frac{2(1-b^2)(1+\cos s)}{(1+b^2-2b\cos s)^2}
\end{equation}
which is non-negative for $b\in[-1,1]$.
If $b\le0$, then we can upper and lower bound the denominator in \eqref{defe} as $1+b^2+2b\le1+b^2-2b\cos s\le1+b^2-2b$ which yields the desired bounds.
For $b\ge0$, the role of the upper and lower bounds is exchanged.
\end{proof}

In order to determine the derivative of the function $\Re(f_0^\g)$ along the contours in Proposition~\ref{prop:gsteep}, we compute the following derivatives.

\begin{lemma}\label{lemma:contourderivatives}
Let $c,r,\gamma\in\R$ be arbitrary.
Then
\begin{align}
\Re\frac{\d}{\d s}\log\left(c+re^{is}\right)&=g\left(-\frac rc,s\right),\label{condourderivative1}\\
\Re\frac{\d}{\d s}\log\left(\gamma(c+re^{is});q\right)_\infty
&=\sum_{k=0}^\infty g\left(\frac{\gamma rq^k}{1-\gamma cq^k},s\right).\label{condourderivative2}
\end{align}
\end{lemma}

\begin{proof}[Proof of Lemma~\ref{lemma:contourderivatives}]
By elementary computations
\begin{equation}
\frac{\d}{\d s}\log\left(c+re^{is}\right)=\frac{irce^{is}+ir^2}{(c+r\cos s)^2+(r\sin s)^2}
=\frac{-\frac rc\sin s+i\frac rc\cos s+i\frac{r^2}{c^2}}{1+\frac{r^2}{c^2}+2\frac rc\cos s}
\end{equation}
where taking real parts gives \eqref{condourderivative1}.
By the definition of the $q$-Pochhammer symbol \eqref{defqPochhammer}
\begin{equation}
\Re\frac{\d}{\d s}\log\left(\gamma(c+re^{is});q\right)_\infty
=\Re\frac{\d}{\d s}\sum_{k=0}^\infty\log\left(1-\gamma cq^k-\gamma rq^ke^{is}\right)
\end{equation}
which yields \eqref{condourderivative2} by applying \eqref{condourderivative1} with $c$ replaced by $1-\gamma cq^k$ and $r$ by $-\gamma rq^k$.
\end{proof}

We give the following series expansion of the $q$-digamma function and that of its derivative which appears in the scaling constants of our models:
\begin{align}
\Psi_q(z)&=-\log(1-q)+\log q\sum_{k=0}^\infty\frac{q^{z+k}}{1-q^{z+k}},\label{psiseries}\\
\Psi_q'(z)&=\left(\log q\right)^2\sum_{k=0}^\infty\frac{q^{z+k}}{(1-q^{z+k})^2}.\label{psi'series}
\end{align}

\begin{lemma}\label{lemma:psimonotonicity}
\begin{enumerate}
\item The function $z\mapsto\Psi_q''(z)/\Psi_q'(z)$ is increasing for positive $z$.
\item The function $\theta\mapsto\kappa_\g(\theta)$ in \eqref{defkappag} is decreasing
whereas $\theta\mapsto\kappa_\p(\theta)$ in \eqref{defkappap} is increasing.
\item For the coefficients in \eqref{defchisigmag} and in \eqref{defchisigmap},
we have that $\chi_\g,\chi_\p,\chi_\a>0$ and $\sigma_\g>0$ for $\theta<\log_q\ami$ and $\sigma_\p>0$ for $\theta>\log_q\ama$.
\end{enumerate}
\end{lemma}

\begin{proof}[Proof of Lemma~\ref{lemma:psimonotonicity}]
\begin{enumerate}
\item
We can essentially repeat a part of the proof of Lemma 4.3 in~\cite{BC16}.
Formula (9) in~\cite{BC16} is the following series representation for the derivatives of the $q$-digamma function:
\begin{equation}\label{psiseriesGB}
\Psi_q^{(k)}(z)=(\log q)^{k+1}\sum_{n=1}^\infty\frac{n^kq^{nz}}{1-q^n}.
\end{equation}
The monotonicity of $z\mapsto\Psi_q''(z)/\Psi_q'(z)$, by checking the derivative, is equivalent to the inequality
\begin{equation}\label{psiderivativeineq}
\Psi_q'''(z)\Psi_q'(z)>\left(\Psi_q''(z)\right)^2
\end{equation}
which, by using the series expansion \eqref{psiseriesGB}, can be written as
\begin{equation}
\sum_{n,m\ge1}\frac{n^3q^{nz}}{1-q^n}\frac{mq^{mz}}{1-q^m}>\sum_{n,m\ge1}\frac{n^2q^{nz}}{1-q^n}\frac{m^2q^{mz}}{1-q^m}.
\end{equation}
If one regards both sides as a power series in $q^z$, then the above inequality certainly holds if the following inequality is true for the $k$th coefficients:
\begin{equation}\label{compareseries}
\sum_{n=1}^{k-1}\frac{n(k-n)}{(1-q^n)(1-q^{k-n})}\frac{n^2+(k-n)^2}2\ge\sum_{n=1}^{k-1}\frac{n(k-n)}{(1-q^n)(1-q^{k-n})}n(k-n)
\end{equation}
which is obtained by symmetrization on the left-hand side.
Since \eqref{compareseries} holds with strict inequality for $k\ge3$, we also have \eqref{psiderivativeineq}
which yields the monotonicity for the $q$-polygamma ratio.

\item
For the derivative
\begin{equation}
\frac{\d}{\d\theta}\kappa_\g
=\frac{\Psi_q''(\theta-\log_qa)\Psi_q'(\theta)-\Psi_q'(\theta-\log_qa)\Psi_q''(\theta)}{(\Psi_q'(\theta))^2}<0
\end{equation}
holds if and only if
\begin{equation}\label{psithetalogaineq}
\frac{\Psi_q''(\theta-\log_qa)}{\Psi_q'(\theta-\log_qa)}<\frac{\Psi_q''(\theta)}{\Psi_q'(\theta)}.
\end{equation}
This however follows from the first statement of this lemma.
Since
\begin{equation}
\frac{\d}{\d\theta}\kappa_\p
=\frac{-\Psi_q''(\log_qa-\theta)\Psi_q'(\log_qa+\theta)-\Psi_q'(\log_qa-\theta)\Psi_q''(\log_qa+\theta)}{(\Psi_q'(\log_qa+\theta))^2}
\end{equation}
for the monotonicity of $\theta\mapsto\kappa_\p(\theta)$ it is enough to use the information about the sign of the $q$-polygamma functions:
$\Psi_q'(z)>0$ and $\Psi_q''(z)<0$ for any $z>0$ and $q\in(0,1)$.

\item
The positivity of $\chi_\g$ follows by comparing the increasing function $\Psi_q''(z)/\Psi_q'(z)$ at $z=\theta-\log_qa$ and at $z=\theta+\log_qa$.
For that of $\chi_\p$ we can use that $\Psi_q''(z)<0$ for $z>0$ and $q\in(0,1)$.

The positivity of $\chi_\a$ can be written equivalently as
\begin{equation}\label{RLPsi}
\log q\,\frac{Rq^\theta-Lq^{-\theta}}{Rq^\theta+Lq^{-\theta}}>\frac{\Psi_q''(\theta)}{\Psi_q'(\theta)}.
\end{equation}
The ratio on the left-hand side $(Rq^\theta-Lq^{-\theta})/(Rq^\theta+Lq^{-\theta})$ is between $-1$ and $1$ for any value of $R,L,\theta\ge0$.
Hence the left-hand side of \eqref{RLPsi} is at least $\log q$.
The right-hand side of \eqref{RLPsi} is however an increasing function of $\theta$ growing
to the limit $\lim_{\theta\to\infty}\Psi_q''(\theta)/\Psi_q'(\theta)=\log q$.
This proves \eqref{RLPsi} and that $\chi_\a>0$.

The positivity of $\sigma_\g$ is equivalent to $\kappa_\g(\theta)>\kappa_\g(\log_q\ami)$
which follows from the monotonicity of $\theta\mapsto\kappa_\g(\theta)$.
Similarly, $\sigma_\p>0$ if and only if $\kappa_\p(\theta)>\kappa_\p(\log_q\ama)$
which is the consequence of the monotonicity of $\theta\mapsto\kappa_\p(\theta)$.
\end{enumerate}
\end{proof}

\begin{lemma}\label{lemma:dominance}
Let $p_0,p_1,p_2,\dots$ and $q_0,q_1,q_2,\dots$ be positive probabilities,
i.e.\ we assume that $\sum_{k=0}^\infty p_k=\sum_{k=0}^\infty q_k=1$.
Suppose further that the inequalities
\begin{equation}\label{decaycompare}
\frac{p_{k+1}}{p_k}\ge\frac{q_{k+1}}{q_k}
\end{equation}
hold for all $k=0,1,2,\dots$.
Then the distribution $(p_k)$ stochastically dominates $(q_k)$, that is, for all $j=0,1,2,\dots$, we have the inequality of tails
\begin{equation}\label{dominance}
\sum_{k=j}^\infty p_k\ge\sum_{k=j}^\infty q_k.
\end{equation}
As a consequence, if $0\le f_0\le f_1\le f_2\le\dots$ is a non-negative increasing function, then for the expectations
\begin{equation}\label{dominanceexpectation}
\sum_{k=0}^\infty f_kp_k\ge\sum_{k=0}^\infty f_kq_k
\end{equation}
follows.
\end{lemma}

\begin{proof}[Proof of Lemma~\ref{lemma:dominance}]
We proceed by induction.
The inequality \eqref{dominance} trivially holds for $j=0$.
Assume that it is true for $j$, we show it for $j+1$.
By \eqref{decaycompare}, we have
\begin{equation}
\frac{\sum_{k=j}^\infty p_k}{p_j}=1+\frac{p_{j+1}}{p_j}+\frac{p_{j+2}}{p_{j+1}}\frac{p_{j+1}}{p_j}+\dots
\ge1+\frac{q_{j+1}}{q_j}+\frac{q_{j+2}}{q_{j+1}}\frac{q_{j+1}}{q_j}+\dots=\frac{\sum_{k=j}^\infty q_k}{q_j}.
\end{equation}
By applying the increasing map $x\mapsto1-1/x$ to both sides, we arrive at the inequality
\begin{equation}\label{comparesumratios}
\frac{\sum_{k=j+1}^\infty p_k}{\sum_{k=j}^\infty p_k}\ge\frac{\sum_{k=j+1}^\infty q_k}{\sum_{k=j}^\infty q_k}.
\end{equation}
The left-hand side of \eqref{comparesumratios} is upper bounded by $\sum_{k=j+1}^\infty p_k/\sum_{k=j}^\infty q_k$ by using the induction hypothesis in the denominator.
The upper bound compared to the right-hand side of \eqref{comparesumratios} immediately yields the dominance \eqref{dominance} for $j+1$.

To prove \eqref{dominanceexpectation}, we use the telescopic decomposition $f_k=f_0+\sum_{j=1}^k(f_j-f_{j-1})$ and we write
\begin{equation}\begin{aligned}
\sum_{k=0}^\infty f_kp_k&=\sum_{k=0}^\infty\left(f_0+\sum_{j=1}^k(f_j-f_{j-1})\right)p_k\\
&=f_0\sum_{k=0}^\infty p_k+\sum_{j=1}^\infty(f_j-f_{j-1})\sum_{k=j}^\infty p_k\\
&\ge f_0\sum_{k=0}^\infty q_k+\sum_{j=1}^\infty(f_j-f_{j-1})\sum_{k=j}^\infty q_k\\
&=\sum_{k=0}^\infty f_kq_k
\end{aligned}\end{equation}
where we exchanged the order of summations in the second equality and we used \eqref{dominance} in the inequality in the third line.
\end{proof}

\subsection{Steep descent properties for geometric $q$-TASEP}
\label{ss:gsteep}

\begin{proof}[Proof of Proposition~\ref{prop:gsteep}]
\begin{enumerate}
\item
By applying Lemma~\ref{lemma:contourderivatives} with $c=a$, $r=-(a-q^\theta)$ and $\gamma=a$ or $1/a$, we have
\begin{multline}\label{f0gderivative}
\Re\frac{\d}{\d s}f_0^\g\left(\log_q(a-(a-q^\theta)e^{is})\right)\\
=\kappa_\g\sum_{k=0}^\infty g\left(-\frac{a(a-q^\theta)q^k}{1-a^2q^k},s\right)
-\sum_{k=1}^\infty g\left(-\frac{(a-q^\theta)q^k}{a(1-q^k)},s\right)
+(f_\g+1)g\left(\frac{a-q^\theta}a,s\right)
\end{multline}
where the $k=0$ term in the second sum on the right-hand side could be removed as it is $g(\infty,s)=0$.
By the fact that $f_0^\g$ has a double critical point at $\theta$, the $s$ derivative of \eqref{f0gderivative} must be $0$ at $s=0$.
Since $\frac{\d}{\d s}g(b,s)|_{s=0}=\frac b{(1-b)^2}$, we conclude that
\begin{equation}\label{ggderivativeeq}
-\kappa_\g\sum_{k=0}^\infty A_k+\sum_{k=1}^\infty B_k+(f_\g+1)C=0
\end{equation}
where
\begin{equation}\label{defABC}
A_k=\frac{\frac{a(a-q^\theta)q^k}{1-a^2q^k}}{\left(1+\frac{a(a-q^\theta)q^k}{1-a^2q^k}\right)^2},\quad
B_k=\frac{\frac{(a-q^\theta)q^k}{a(1-q^k)}}{\left(1+\frac{(a-q^\theta)q^k}{a(1-q^k)}\right)^2},\quad
C=\frac{\frac{a-q^\theta}a}{\left(1-\frac{a-q^\theta}a\right)^2}
\end{equation}
are positive coefficients.

The notation above allows us to rewrite the derivative in \eqref{f0gderivative} using the function $h$ as
\begin{multline}\label{f0gderivative2}
\Re\frac{\d}{\d s}f_0^\g\left(\log_q(a-(a-q^\theta)e^{is})\right)
=-\kappa_\g\sum_{k=0}^\infty A_kh\left(-\frac{a(a-q^\theta)q^k}{1-a^2q^k},s\right)\\
+\sum_{k=1}^\infty B_kh\left(-\frac{(a-q^\theta)q^k}{a(1-q^k)},s\right)
+(f_\g+1)Ch\left(\frac{a-q^\theta}a,s\right).
\end{multline}

In order to show that the derivative in \eqref{f0gderivative} or in \eqref{f0gderivative2} is non-positive for $s\in(0,\pi)$,
one has to see that the first sum on the right-hand side of \eqref{f0gderivative2} which is non-positive for $s\in(0,\pi)$
compensates the second sum and the third term which are non-negative.
The idea is that the second sum is completely compensated by the $k=0$ term of the first sum and the third term is compensated by the rest of the first sum.

Since $k\mapsto q^k/(1-q^k)$ is decreasing, the sequence of first arguments of $h$ in absolute value in the second sum on the right-hand side
of \eqref{f0gderivative2} is also decreasing in $k$.
We will see that this argument in the $k=1$ and hence in all terms are less than $1$ in absolute value.
The condition for this to hold is $(a-q^\theta)q/(a(1-q))<1$ and it can also be written equivalently as $2a-q^\theta<aq^{-1}$
which is exactly the condition for the poles at $W=\log_qa-1,\log_qa-2,\dots$ coming from the factor $(q^W/a;q)_\infty$
in the kernel $K^\g_z$ in \eqref{defKgz} to remain outside of $\C_\theta$.
This follows from \eqref{gthetacond}, see also Step 1 in the proof of Proposition~\ref{prop:glocalization}.

For the $k=0$ term in the first sum on the right-hand side of \eqref{f0gderivative2} to compensate alone the second sum using Lemma~\ref{lemma:hcompare},
the following conditions have to be satisfied.
Since $g(1/b,s)=g(b,s)$, if both
\begin{align}
\frac{a(a-q^\theta)}{1-a^2}&\ge\frac{(a-q^\theta)q}{a(1-q)},\label{gargumentcompare}\\
\frac{1-a^2}{a(a-q^\theta)}&\ge\frac{(a-q^\theta)q}{a(1-q)}\label{gargumentcompare2}
\end{align}
hold, then
\begin{equation}\label{ghbound}
h\left(-\frac{a(a-q^\theta)}{1-a^2},s\right)\ge h\left(\frac{(a-q^\theta)q}{a(1-q)},s\right)\ge\dots
\ge h\left(\frac{(a-q^\theta)q^k}{a(1-q^k)},s\right)\ge\dots
\end{equation}
follows by Lemma~\ref{lemma:hcompare} for $s\in[0,\pi]$.
Hence
\begin{equation}\label{f0gderivative3}
-\left(\sum_{k=1}^\infty B_k\right)h\left(-\frac{a(a-q^\theta)}{1-a^2},s\right)
+\sum_{k=1}^\infty B_kh\left(-\frac{(a-q^\theta)q^k}{a(1-q^k)},s\right)\le0
\end{equation}
is obtained by a combination of \eqref{ghbound} for $s\in(0,\pi)$.
The condition \eqref{gargumentcompare} is equivalent to $q\le a^2$ which is assumed as the lower bound in \eqref{gainterval}.
The condition \eqref{gargumentcompare2} is equivalent to
\begin{equation}\label{gaqthetasqrt}
a-q^\theta\le\sqrt{(1-a^2)\frac{1-q}q}.
\end{equation}
The upper bound in \eqref{gainterval} however implies $a^2\le1-q/(1-q)^2\le1-q/(1-q)$.
For these values of $a$ and $q$, the right-hand side of \eqref{gaqthetasqrt} is at least $1$,
but the left-hand side is less than $1$ as a difference of two numbers in $(0,1)$, hence \eqref{gaqthetasqrt} is always satisfied.

On the other hand,
\begin{equation}\label{ghbound2}
h\left(-\frac{a(a-q^\theta)q^k}{1-a^2q^k},s\right)\ge h\left(\frac{a-q^\theta}a,s\right)
\end{equation}
for $s\in(0,\pi)$ is a trivial consequence of Lemma~\ref{lemma:hcompare} for $k=0,1,2,\dots$,
since the first arguments on the two sides above have different signs.
Suppose that
\begin{equation}\label{gcoeffcond}
\kappa_\g A_0\ge\sum_{k=1}^\infty B_k
\end{equation}
holds.
Then a combination of \eqref{ghbound2} is the inequality
\begin{multline}\label{f0gderivative4}
-\left(\kappa_\g A_0-\sum_{k=1}^\infty B_k\right)h\left(-\frac{a(a-q^\theta)}{1-a^2},s\right)
-\kappa_\g\sum_{k=1}^\infty A_kh\left(-\frac{a(a-q^\theta)q^k}{1-a^2q^k},s\right)\\
+\left(\kappa_\g\sum_{k=0}^\infty A_k-\sum_{k=1}^\infty B_k\right)h\left(\frac{a-q^\theta}a,s\right)\le0
\end{multline}
for $s\in(0,\pi)$.
It follows from \eqref{ggderivativeeq} that the coefficient of $h$ in the last term on the left-hand side of \eqref{f0gderivative4} is exactly $f_\g+1$.
By this observation, adding the inequalities \eqref{f0gderivative3} and \eqref{f0gderivative4} implies that \eqref{f0gderivative2} is non-positive for $s\in(0,\pi)$ which is the steep descent property in the first part of the proposition.

We are left with proving the inequality \eqref{gcoeffcond}.
It can be written by the definitions \eqref{defABC} and after simplification as
\begin{equation}\label{gcoeffcondrewrite}
\kappa_\g\frac{1-a^2}{(1-aq^\theta)^2}\ge\sum_{k=1}^\infty\frac{(1-q^k)q^k}{(a-q^{\theta+k})^2}.
\end{equation}
Since $\kappa_\g$ and $1/(1-cq^\theta)$ are both decreasing functions of $\theta$ for any $c>0$ (see Lemma~\ref{lemma:psimonotonicity}),
both sides of \eqref{gcoeffcondrewrite} are decreasing in $\theta$.
Hence we compare the $\theta\to\infty$ limit of the left-hand side of \eqref{gcoeffcondrewrite} with the $\theta=\log_qa$ value of the right-hand side.
By the series expansion \eqref{psi'series}, one can see that $\kappa_\g\to1/a^2$ as $\theta\to\infty$,
and the left-hand side of \eqref{gcoeffcondrewrite} goes to $(1-a^2)/a^2$ as $\theta\to\infty$.
The right-hand side of \eqref{gcoeffcondrewrite} at $\theta=\log_qa$ is equal to
\begin{equation}\label{geometricsumbound}
\frac1{a^2}\sum_{k=1}^\infty\frac{q^k}{1-q^k}\le\frac1{a^2(1-q)}\sum_{k=1}^\infty q^k=\frac q{a^2(1-q)^2}.
\end{equation}
Hence we get that if
\begin{equation}\label{gcond16}
\frac{1-a^2}{a^2}\ge\frac q{a^2(1-q)^2},
\end{equation}
then \eqref{gcoeffcondrewrite} and \eqref{gcoeffcond} hold.
But \eqref{gcond16} is equivalent to the upper bound in \eqref{gainterval}, hence the first part of the proposition is proved.

\item
We apply Lemma~\ref{lemma:contourderivatives} with $c=0$, $\gamma=1$ and $r=q^\theta$ or $q^\theta/a$ to get that
\begin{equation}
-\Re\frac{\d}{\d t}f_0^\g\left(\theta+i\frac t{\log q}\right)
=-\kappa_\g\sum_{k=0}^\infty g\left(aq^{\theta+k},t\right)+\sum_{k=0}^\infty g\left(\frac{q^{\theta+k}}a,t\right).
\end{equation}
For the proof in this part, we have to see that the derivative above is non-positive for $t\in(0,\pi)$.
Rewriting the function $g$ in terms of $h$ by \eqref{defh}, $\kappa_\g$ according to its definition \eqref{defkappag} and the series expansion \eqref{psi'series} yields that we have to show the inequality
\begin{equation}\label{ghexpcompare}
\frac{\sum_{k=0}^\infty\frac{aq^{\theta+k}}{(1-aq^{\theta+k})^2}h\left(aq^{\theta+k},t\right)}
{\sum_{k=0}^\infty\frac{aq^{\theta+k}}{(1-aq^{\theta+k})^2}}
\ge\frac{\sum_{k=0}^\infty\frac{q^{\theta+k}/a}{(1-q^{\theta+k}/a)^2}h\left(q^{\theta+k}/a,t\right)}
{\sum_{k=0}^\infty\frac{q^{\theta+k}/a}{(1-q^{\theta+k}/a)^2}}
\end{equation}
for $t\in(0,\pi)$.

We use Lemma~\ref{lemma:dominance} with
\begin{equation}
p_k=\frac{\frac{aq^{\theta+k}}{(1-aq^{\theta+k})^2}}{\sum_{j=0}^\infty\frac{aq^{\theta+j}}{(1-aq^{\theta+j})^2}},\qquad
q_k=\frac{\frac{q^{\theta+k}/a}{(1-q^{\theta+k}/a)^2}}{\sum_{j=0}^\infty\frac{q^{\theta+j}/a}{(1-q^{\theta+j}/a)^2}}
\end{equation}
which are clearly probability distributions, i.e.\ they sum up to $1$.
The ratio of two consecutive weights is
\begin{equation}\label{pqratio}
\frac{p_{k+1}}{p_k}=q\frac{(1-aq^{\theta+k})^2}{(1-aq^{\theta+k+1})^2},\qquad
\frac{q_{k+1}}{q_k}=q\frac{(1-q^{\theta+k}/a)^2}{(1-q^{\theta+k+1}/a)^2}
\end{equation}
for the two sequences.
The derivative of the function $x\mapsto(1-Dx)/(1-Ex)$ is $(E-D)/(1-Ex)^2$ and it is negative for $D=q^{\theta+k}$ and $E=q^{\theta+k+1}$.
The function $x\mapsto(1-q^{\theta+k}x)/(1-q^{\theta+k+1}x)$ is decreasing, hence \eqref{decaycompare} holds for \eqref{pqratio}.
By the fact that for $t\in(0,\pi)$ the function $k\mapsto h(aq^{\theta+k},t)$ is increasing (see Lemma~\ref{lemma:hcompare}), we conclude by \eqref{dominanceexpectation} of Lemma~\ref{lemma:dominance} that
\begin{equation}
\sum_{k=0}^\infty p_kh\left(aq^{\theta+k},t\right)\ge\sum_{k=0}^\infty q_kh\left(aq^{\theta+k},t\right)
\ge\sum_{k=0}^\infty q_kh\left(q^{\theta+k}/a,t\right)
\end{equation}
where the last inequality follows because $h(aq^{\theta+k},t)\ge h(q^{\theta+k}/a,t)$ for all $k=0,1,2,\dots$
which proves \eqref{ghexpcompare} and the proposition.
\end{enumerate}
\end{proof}

\begin{proof}[Proof of Proposition~\ref{prop:gsteep2}]
\begin{enumerate}
\item
First we show that the circular part of $\C_\theta^{\log_q\ami,\delta}$, that is $\C_\theta$ is of steep descent for the function $\Re(g_0^\g)$.
By Proposition~\ref{prop:gsteep}, the contour $\C_\theta$ is of steep descent for $\Re(f_0^\g)$.
Comparing the functions $g_0^\g$ and $f_0^\g$ gives that
\begin{equation}\label{g0f0}\begin{aligned}
g_0^\g(W)-f_0^\g(W)&=(g_\g-f_\g)\log qW\\
&=\big[\kappa_\g\left(\Psi_q(\log_q\ami+\log_qa)-\Psi_q(\theta+\log_qa)\right)\\
&\qquad-\Psi_q(\log_q\ami-\log_qa)+\Psi_q(\theta-\log_qa)\big]W.
\end{aligned}\end{equation}
Next we show that the coefficient of $W$ on the right-hand side of \eqref{g0f0} is positive.
By the increasing property of $z\mapsto\Psi_q(z)$ for $z>0$, we have that $\Psi_q(\log_q\ami+\log_qa)-\Psi_q(\theta+\log_qa)>0$,
hence the positivity of the coefficient of $W$ on the right-hand side of \eqref{g0f0} can be written equivalently as
\begin{equation}\label{g0f0coeff}
\kappa_\g>\frac{\Psi_q(\log_q\ami-\log_qa)-\Psi_q(\theta-\log_qa)}{\Psi_q(\log_q\ami+\log_qa)-\Psi_q(\theta+\log_qa)}
=\frac{\Psi_q'(x-\log_qa)}{\Psi_q'(x+\log_qa)}
\end{equation}
for some $x\in(\theta,\log_q\ami)$ where the equality above follows by Cauchy's mean value theorem.
By the definition \eqref{defkappag}, one then recognizes $\kappa_\g$ with $\theta$ replaced by $x$ on the right-hand side of \eqref{g0f0coeff}.
Hence the inequality in \eqref{g0f0coeff} follows by the decreasing property of $\kappa_\g$ proved in Lemma~\ref{lemma:psimonotonicity}.

The derivative of the difference $\Re(g_0^\g-f_0^\g)$ along $\C_\theta$ is
\begin{equation}\begin{aligned}
&\Re\frac{\d}{\d s}\left(g_0^\g-f_0^\g\right)\left(\log_q(a-(a-q^\theta)e^{is})\right)\\
&\qquad=\big[\kappa_\g\left(\Psi_q(\log_q\ami+\log_qa)-\Psi_q(\theta+\log_qa)\right)\\
&\qquad\qquad-\Psi_q(\log_q\ami-\log_qa)+\Psi_q(\theta-\log_qa)\big]
\frac1{\log q}g\left(\frac{a-q^\theta}a,s\right)
\end{aligned}\end{equation}
by \eqref{g0f0} and Lemma~\ref{lemma:contourderivatives}.
The function $g\left(\frac{a-q^\theta}a,s\right)$ is positive for $s\in(0,\pi)$, but its coefficient is negative due to the $1/\log q$ factor.
This proves the steep descent property along the circular part of $\C_\theta^{\log_q\ami,\delta}$.

For the V-shaped part of the contour, we consider the Taylor expansion \eqref{g0Taylorg} of $g_0^\g$ around $\log_q\ami$
and we choose $\delta>0$ so small that the steep descent property holds along $V_{\log_q\ami,19\pi/24}^\delta$.
The angle $\pm\frac{19\pi}{24}$ of $W-\log_q\ami$ is chosen so that $\frac{3\pi}4<\frac{19\pi}{24}<\frac{5\pi}6$,
hence both $\Re((W-\log_q\ami)^2)$ and $\Re((W-\log_q\ami)^3)$ are positive.

The coefficient of the quadratic term in the expansion \eqref{g0Taylorg} is $-\sigma_\g/2$ which is negative by Lemma~\ref{lemma:psimonotonicity},
but it goes to $0$ as $\theta\to\log_q\ami$ hence we cannot use it to compensate the error term of the Taylor approximation.
The coefficient of the cubic term is $-\xi_\g/3$ which is a continuous function of $\theta$, it coincides with $-\chi_\g/3$ for $\theta=\log_q\ami$
hence it is negative for $\theta=\log_q\ami$ by Lemma~\ref{lemma:psimonotonicity}.
As a consequence the coefficient $-\xi_\g/3$ is negative also in a neighbourhood of $\theta=\log_q\ami$.

The error of the Taylor approximation in \eqref{g0Taylorg} can be upper bounded by $M|W-\log_q\ami|^4/4!$
with $M$ being the absolute supremum of the fourth derivative of $g_0^\g$ in the neighbourhood of $\log_q\ami$ in which we consider the expansion.
Since $(g_0^\g)^{(4)}(\theta)=-\kappa_\g\Psi_q'''(\theta+\log_qa)+\Psi_q'''(\theta-\log_qa)$ which is a continuous function of $\theta$,
its supremum is finite in a neighbourhood of $\theta=\log_q\ami$.
Hence the cubic term in the approximation compensates the error if the neighbourhood is chosen to be small enough.
This determines the choice of $\delta$ depending on $q,a,\ami$
for which the contour $V_{\log_q\ami,19\pi/24}^\delta$ has the steep descent property for $\Re(g_0^\g)$.
Given $\delta>0$, let $\varepsilon$ be the determined by the property that for $\theta=\log_q\ami-\varepsilon$ the contour $\C_\theta$
passes through the two endpoints of $V_{\log_q\ami,19\pi/24}^\delta$.

\item
By Proposition~\ref{prop:gsteep} for $\theta$ replaced by $\log_q\ami$,
the contour $\D_{\log_q\ami}$ is of steep descent for the function $-\Re(\wt f_0^\g)$ where
\begin{equation}
\wt f_0^\g(W)=\kappa_\g(\log_q\ami)\log(aq^W;q)_\infty-\log(q^W/a;q)_\infty+(f_\g(\log_q\ami)+1)\log q\,W
\end{equation}
with $\kappa_\g(\log_q\ami)$ and $f_\g(\log_q\ami)$ being the constants defined in \eqref{defkappag} and \eqref{deffg}
with $\theta$ replaced by $\log_q\ami$.
Comparing the functions $g_0^\g$ and $\wt f_0^\g$ gives that
\begin{multline}\label{g0f0tilde}
g_0^\g(W)-\wt f_0^\g(W)
=\left(\kappa_\g(\theta)-\kappa_\g(\log_q\ami)\right)\\
\times\left[\log(aq^W;q)_\infty+(\Psi_q(\log_q\ami+\log_qa)+\log(1-q))W\right]
\end{multline}
where $\kappa_\g(\theta)-\kappa_\g(\log_q\ami)>0$ by Lemma~\ref{lemma:psimonotonicity}.
By taking the derivative of the difference in \eqref{g0f0tilde} along $\D_{\log_q\ami}$ we find that
\begin{equation}
-\frac1{\kappa_\g(\theta)-\kappa_\g(\log_q\ami)}\Re\frac{\d}{\d s}\left(g_0^\g-\wt f_0^\g\right)\left(\log_q\ami+i\frac t{\log q}\right)
=-\sum_{k=0}^\infty g(a\ami q^k,t)
\end{equation}
which is negative for $t\in(0,\pi)$.
\end{enumerate}
\end{proof}

\subsection{Steep descent properties for geometric $q$-PushTASEP}
\label{ss:psteep}

\begin{proof}[Proof of Proposition~\ref{prop:psteep}]
\begin{enumerate}
\item
First we rewrite the second $q$-Pochhammer symbol in $f_0^\p$ in \eqref{deff0p} along $W\in\wt\C_\theta$, that is for $W(s)=\log_q(a+(q^\theta-a)e^{is})$ as
\begin{equation}
\log(aq^{-W};q)_\infty=\sum_{k=0}^\infty\left(\log(a(1-q^k)+(q^\theta-a)e^{is})-\log(a+(q^\theta-a)e^{is})\right).
\end{equation}
Then Lemma~\ref{lemma:contourderivatives} yields
\begin{multline}\label{f0pderivative}
\Re\frac{\d}{\d s}f_0^\p\left(\log_q(a+(q^\theta-a)e^{is})\right)
=\kappa_\p\sum_{k=0}^\infty g\left(\frac{a(q^\theta-a)q^k}{1-a^2q^k},s\right)\\
-\sum_{k=0}^\infty\left(g\left(-\frac{q^\theta-a}{a(1-q^k)},s\right)-g\left(-\frac{q^\theta-a}a,s\right)\right)
+(f_\p-1)g\left(-\frac{q^\theta-a}a,s\right).
\end{multline}
By the fact that $f_0^\p$ has a double critical point and by comparing the $s$ derivatives on the right-hand side of \eqref{f0pderivative}, we have
\begin{equation}\label{psumcoeff}
\kappa_\p\sum_{k=0}^\infty A_k+\sum_{k=1}^\infty(B_k-C)-f_\p C=0
\end{equation}
where
\begin{equation}\label{defABCp}
A_k=\frac{\frac{a(q^\theta-a)q^k}{1-a^2q^k}}{\left(1-\frac{a(q^\theta-a)q^k}{1-a^2q^k}\right)^2},\quad
B_k=\frac{\frac{q^\theta-a}{a(1-q^k)}}{\left(1+\frac{q^\theta-a}{a(1-q^k)}\right)^2},\quad
C=\frac{\frac{q^\theta-a}a}{\left(1+\frac{q^\theta-a}a\right)^2}
\end{equation}
are positive coefficients different from the ones defined in \eqref{defABC} in the case of the geometric $q$-TASEP.
We also used the fact that $g(\infty,s)=0$ appears in the $k=0$ in the second sum on the right-hand side of \eqref{f0pderivative}.

For the $k=1,2,\dots$ terms in the second sum on the right-hand side of \eqref{f0pderivative},
we use an idea similar to the proof of Lemma 6.5 in~\cite{Vet15}.
We apply Cauchy's mean value theorem on the interval $[b_k,c]$ with $b_k=-(q^\theta-a)/(a(1-q^k))$ and $c=-(q^\theta-a)/a$ to get that
\begin{equation}\label{cauchymeanvalue}
\frac{g(c,s)-g(b_k,s)}{\frac c{(1-c)^2}-\frac{b_k}{(1-b_k)^2}}=\frac{\frac\d{\d b}g(b,s)|_{b=x}}{\frac\d{\d b}\frac b{(1-b)^2}|_{b=x}}
=(h(x,s))^2\frac1{\sin s}
\end{equation}
with some $x\in(b_k,c)$ where the last equality in \eqref{cauchymeanvalue} follows by direct computation.
Note that $b_1<b_2<\dots<c$ and that the condition \eqref{pcondpole} is equivalent to $-1\le(q^\theta-a)/(a(1-q))=b_1$.
Hence all the $b_k$s for $k=1,2,\dots$ and $c$ lie in the interval $[-1,0]$.
One can check that the functions $b\mapsto g(b,s)$ for $s\in(0,\pi)$ and $b\mapsto b/(1-b)^2$ are both increasing on $b\in[-1,0]$.
As a consequence, both the numerator and the denominator on the left-hand side of \eqref{cauchymeanvalue} are positive.
By Lemma~\ref{lemma:hcompare}, we can upper bound the right-hand side of \eqref{cauchymeanvalue}
by using the inequality $h(x,s)\le h(b_1,s)$ for $s\in(0,\pi)$ because $b_1=\inf\{(b_k,c),k\ge1\}$.
After multiplying by the denominator $c/(1-c)^2-b_k/(1-b_k)^2=B_k-C>0$
on the left-hand side of \eqref{cauchymeanvalue}, we get that
\begin{equation}\label{ccauchybound}
g\left(-\frac{q^\theta-a}a,s\right)-g\left(-\frac{q^\theta-a}{a(1-q^k)},s\right)
\le\left(B_k-C\right)h\left(-\frac{q^\theta-a}{a(1-q)},s\right)^2\frac1{\sin s}
\end{equation}
for $s\in(0,\pi)$.

On the other hand, by the monotonicity of $h$ in the first variable (see Lemma~\ref{lemma:hcompare}) and by the property $h(1/b,s)=h(b,s)$,
we have the inequality
\begin{equation}\label{phcompare}
h\left(\frac{a(q^\theta-a)q^k}{1-a^2q^k},s\right)\le h(0,s)=\sin s
\end{equation}
for $k=0,1,2,\dots$ for $s\in(0,\pi)$.
By rewriting the right-hand side of \eqref{f0pderivative} and then using \eqref{ccauchybound} and \eqref{phcompare}, we get that
\begin{multline}\label{f0pderivative2}
\Re\frac{\d}{\d s}f_0^\p\left(\log_q(a+(q^\theta-a)e^{is})\right)\\
\le\kappa_\p\sum_{k=0}^\infty A_k\sin s
+\sum_{k=1}^\infty\left(B_k-C\right)h\left(-\frac{q^\theta-a}{a(1-q)},s\right)^2\frac1{\sin s}-f_\p Ch\left(-\frac{q^\theta-a}a,s\right).
\end{multline}

By recalling the notation $b_1=-(q^\theta-a)/(a(1-q))$ and $c=-(q^\theta-a)/a$,
we can rewrite two terms containing the function $h$ on the right-hand side of \eqref{f0pderivative2} using \eqref{hwithe} as
\begin{align}
\frac{h(b_1,s)^2}{\sin s}&=\left(1-\frac{4b_1(1-\cos s)}{(1+b_1)^2}e(b_1,s)+\frac{4b_1^2(1-\cos s)^2}{(1+b_1)^4}e(b_1,s)^2\right)\sin s,\label{hb1}\\
h(c,s)&=\left(1-\frac{2c(1-\cos s)}{(1+c)^2}e(c,s)\right)\sin s.\label{hc}
\end{align}
Next we substitute \eqref{hb1}--\eqref{hc} on the right-hand side of \eqref{f0pderivative2} and
we separate the terms containing $\sin s$ only and those containing the function $e$.
The sum of the coefficients of the terms containing only $\sin s$ is $0$ by \eqref{psumcoeff}.
Hence for the non-negativity of the derivative in \eqref{f0pderivative2} for $s\in(0,\pi)$ it suffices to show that
\begin{equation}\label{pcondtech2}
\sum_{k=1}^\infty\left(B_k-C\right)\left(\frac{-4b_1}{(1+b_1)^2}e(b_1,s)+\frac{4b_1^2(1-\cos s)}{(1+b_1)^4}e(b_1,s)^2\right)
\le f_\p C\frac{-2c}{(1+c)^2}e(c,s)
\end{equation}
after simplifying by a factor $(1-\cos s)$.

In order to show \eqref{pcondtech2}, we upper bound its left-hand side and lower bound its right-hand side.
First we deal with $B_k-C$ defined in \eqref{defABCp} for which direct computations yield that
\begin{multline}
B_k-C=\frac{aq^k(q^\theta-a)(q^{2\theta}-2aq^\theta+a^2q^k)}{q^{2\theta}(q^\theta-aq^k)^2}
\le\frac{a(q^\theta-a)}{q^{2\theta}}q^k\frac{q^{2\theta}-2aq^\theta+a^2}{(q^\theta-a)^2}
\end{multline}
where the inequality above follows from applying $q^k\le1$ twice, hence
\begin{equation}\label{sumBkCbound}
\sum_{k=1}^\infty(B_k-C)\le\sum_{k=1}^\infty\frac{a(q^\theta-a)}{q^{2\theta}}q^k=\frac{aq(q^\theta-a)}{q^{2\theta}(1-q)}.
\end{equation}
To bound the other factor on the left-hand side of \eqref{pcondtech2},
we use that $e(b_1,s)^2\le e(b_1,s)\le e(c,s)$ by Lemma~\ref{lemma:eproperties} and the inequality $1-\cos s\le2$.
We get that
\begin{equation}\label{2eterm}\begin{aligned}
&\frac{-4b_1}{(1+b_1)^2}e(b_1,s)+\frac{4b_1^2(1-\cos s)}{(1+b_1)^4}e(b_1,s)^2\\
&\qquad\le\left(\frac{-4b_1}{(1+b_1)^2}+\frac{8b_1^2}{(1+b_1)^4}\right)e(c,s)\\
&\qquad=\frac{4a(1-q)(q^\theta-a)\left(a^2(1-q)^2+(q^\theta-a)^2\right)}{(2a-q^\theta-aq)^4}e(c,s).
\end{aligned}\end{equation}

On the right-hand side of \eqref{pcondtech2}, we have
\begin{equation}\label{Cccompute}
C=\frac{a(q^\theta-a)}{q^{2\theta}},\qquad\frac{-2c}{(1+c)^2}=\frac{2a(q^\theta-a)}{(2a-q^\theta)^2}.
\end{equation}
Putting together the bounds \eqref{sumBkCbound} and \eqref{2eterm} and the equalities in \eqref{Cccompute},
we obtain that after simplifications the inequality \eqref{pcondtech2} is implied by the condition \eqref{pcondtech}.
This yields the required steep descent property.

\item
Lemma~\ref{lemma:contourderivatives} with $c=0$, $\gamma=1$ and $r=aq^\theta$ and the $\Re(f_0^\p(\overline w))=\Re(f_0^\p(w))$ symmetry implies that
\begin{equation}
-\Re\frac{\d}{\d t}f_0^\p\left(\theta+i\frac t{\log q}\right)
=-\kappa_\p\sum_{k=0}^\infty g\left(aq^{\theta+k},t\right)+\sum_{k=0}^\infty g\left(aq^{-\theta+k},t\right).
\end{equation}
As in the proof of Proposition~\ref{prop:gsteep}, the non-positivity of the derivative above for $t\in(0,\pi)$ is equivalent to
\begin{equation}
\frac{\sum_{k=0}^\infty\frac{aq^{\theta+k}}{(1-aq^{\theta+k})^2}h\left(aq^{\theta+k},t\right)}
{\sum_{k=0}^\infty\frac{aq^{\theta+k}}{(1-aq^{\theta+k})^2}}
\ge\frac{\sum_{k=0}^\infty\frac{aq^{-\theta+k}}{(1-aq^{-\theta+k})^2}h\left(aq^{-\theta+k},t\right)}
{\sum_{k=0}^\infty\frac{aq^{-\theta+k}}{(1-aq^{-\theta+k})^2}}
\end{equation}
for $t\in(0,\pi)$ which has the same proof as \eqref{ghexpcompare}.
\end{enumerate}
\end{proof}

\begin{proof}[Proof of Proposition~\ref{prop:psteep2}]
\begin{enumerate}
\item
First we consider the circular part of $\wt\C_\theta^{\,\log_q\ama,\delta}$
By Proposition~\ref{prop:psteep}, $\wt\C_\theta$ is of steep descent for $\Re(f_0^\p)$.
Comparing the functions $g_0^\p$ and $f_0^\p$ yields
\begin{equation}\label{g0f0p}\begin{aligned}
g_0^\p(W)-f_0^\p(W)&=(g_\p-f_\p)\log qW\\
&=\big[\kappa_\p\left(\Psi_q(\log_qa+\log_q\ama)-\Psi_q(\log_qa+\theta)\right)\\
&\qquad+\Psi_q(\log_qa-\log_q\ama)-\Psi_q(\log_qa-\theta)\big]W.
\end{aligned}\end{equation}
By the increasing property of $z\mapsto\Psi_q(z)$ for $z>0$, $\Psi_q(\log_qa+\log_q\ama)-\Psi_q(\log_qa+\theta)<0$,
hence the negativity of the coefficient of $W$ on the right-hand side of \eqref{g0f0p} is equivalent to
\begin{equation}\label{g0f0coeffp}
\kappa_\p>\frac{-\Psi_q(\log_qa-\log_q\ama)+\Psi_q(\log_qa-\theta)}{\Psi_q(\log_qa+\log_q\ama)-\Psi_q(\log_qa+\theta)}
=\frac{\Psi_q'(\log_qa-x)}{\Psi_q'(\log_qa+x)}
\end{equation}
for some $x\in(\log_q\ama,\theta)$ by Cauchy's mean value theorem.
By \eqref{defkappap}, the right-hand side of \eqref{g0f0coeffp} is equal to $\kappa_\p$ with $\theta$ replaced by $x$
and \eqref{g0f0coeffp} follows by the increasing property of $\kappa_\p$, see Lemma~\ref{lemma:psimonotonicity}.

The derivative of the difference $\Re(g_0^\p-f_0^\p)$ along $\wt\C_\theta$ is
\begin{equation}\begin{aligned}
&\Re\frac{\d}{\d s}\left(g_0^\p-f_0^\p\right)\left(\log_q(a+(q^\theta-a)e^{is})\right)\\
&\qquad=\big[\kappa_\p\left(\Psi_q(\log_qa+\log_q\ama)-\Psi_q(\log_qa+\theta)\right)\\
&\qquad\qquad+\Psi_q(\log_qa-\log_q\ama)-\Psi_q(\log_qa-\theta)\big]
\frac1{\log q}g\left(-\frac{q^\theta-a}a,s\right)
\end{aligned}\end{equation}
by \eqref{g0f0p} and Lemma~\ref{lemma:contourderivatives}.
The function $g\left(-\frac{q^\theta-a}a,s\right)$ is negative for $s\in(0,\pi)$,
hence the steep descent property along the circular part of $\wt\C_\theta^{\,\log_q\ama,\delta}$ follows.

On the V-shaped part $\Re((W-\log_q\ama)^2)>0$ and $\Re((W-\log_q\ama)^3)<0$ in the Taylor expansion \eqref{g0Taylorp}.
The coefficient of the quadratic term $-\sigma_\p/2$ is negative by Lemma~\ref{lemma:psimonotonicity}.
That of the cubic term $\xi_\p/3$ is positive in a neighbourhood of $\theta=\log_q\ama$
as a continuous function of $\theta$ which coincides with $\chi_\p/3>0$ for $\theta=\log_q\ama$.
The error term of the Taylor approximation is bounded by the cubic term in a neighbourhood of $\log_q\ama$
in the same way as in the proof of Proposition~\ref{prop:gsteep2}.
This yields that $V_{\log_q\ama,5\pi/24}^\delta$ has the steep descent property for $\Re(g_0^\p)$ with some small enough $\delta>0$
and $\varepsilon$ is such that the endpoints of the V-shaped contour lie on $\wt\C_\theta$ for $\theta=\log_q\ama+\varepsilon$.

\item
By Proposition~\ref{prop:gsteep}, $\D_{\log_q\ama}$ is a steep descent contour for the function $-\Re(\wt f_0^\p)$ with
\begin{equation}
\wt f_0^\p(W)=\kappa_\p(\log_q\ama)\log(aq^W;q)_\infty-\log(aq^{-W};q)_\infty+(f_\p(\log_q\ama)-1)\log q\,W.
\end{equation}
Comparing $g_0^\p$ and $\wt f_0^\p$ yields
\begin{multline}\label{g0f0tildep}
g_0^\p(W)-\wt f_0^\p(W)
=\left(\kappa_\p(\theta)-\kappa_\p(\log_q\ama)\right)\\
\times\left[\log(aq^W;q)_\infty+(\Psi_q(\log_qa+\log_q\ama)+\log(1-q))W\right]
\end{multline}
with $\kappa_\p(\theta)-\kappa_\p(\log_q\ama)>0$ by Lemma~\ref{lemma:psimonotonicity}.
For its derivative along $\D_{\log_q\ama}$ we have
\begin{multline}
-\frac1{\kappa_\p(\theta)-\kappa_\p(\log_q\ama)}\Re\frac{\d}{\d s}\left(g_0^\p-\wt f_0^\p\right)\left(\log_q\ama+i\frac t{\log q}\right)\\
=-\sum_{k=0}^\infty g(a\ama q^k,t)
\end{multline}
which is negative for $t\in(0,\pi)$.
\end{enumerate}
\end{proof}

\subsection{Steep descent properties for $q$-PushASEP}
\label{ss:asteep}

\begin{proof}[Proof of Proposition~\ref{prop:asteep}]
\begin{enumerate}
\item
A direct computation yields that
\begin{equation}\label{athirdterm}
\Re\frac{\d}{\d s}q^{-\log_q(1-re^{is})}=-\frac{r(1-r^2)\sin s}{(1+r^2-2r\cos s)^2}
\end{equation}
which is the derivative of the third term of $\Re(f_0^\a)$ along $\ol\C_\theta$ with $r=1-q^\theta$.
Using Lemma~\ref{lemma:contourderivatives} and \eqref{athirdterm} and the fact that $g(\infty,s)=0$ we have
\begin{equation}\label{f0aderivative}\begin{aligned}
&-\Re\frac{\d}{\d s}f_0^\a\left(\log_q(1-(1-q^\theta)e^{is})\right)\\
&\quad=-\sum_{k=1}^\infty g\left(-\frac{(1-q^\theta)q^k}{1-q^k},s\right)-\kappa_\a R(1-q^\theta)\sin s\\
&\qquad+\kappa_\a L\frac{(1-q^\theta)(1-(1-q^\theta)^2)\sin s}{(1+(1-q^\theta)^2-2(1-q^\theta)\cos s)^2}
+(f_\a+1)g\left(1-q^\theta,s\right)\\
&\quad=\sum_{k=1}^\infty A_k h\left(-\frac{(1-q^\theta)q^k}{1-q^k},s\right)-\kappa_\a R(1-q^\theta)h(0,s)\\
&\qquad+\kappa_\a L\frac{(1-q^\theta)(1-(1-q^\theta)^2)}{q^{4\theta}}\frac{h(1-q^\theta,s)^2}{\sin s}+(f_\a+1)\frac{1-q^\theta}{q^{2\theta}}h(1-q^\theta,s)
\end{aligned}\end{equation}
where the last equality follows by writing all terms using the function $h$ defined in \eqref{defh} and
\begin{equation}
A_k=\frac{\frac{(1-q^\theta)q^k}{1-q^k}}{\left(1+\frac{(1-q^\theta)q^k}{1-q^k}\right)^2}=\frac{(1-q^\theta)q^k(1-q^k)}{(1-q^{\theta+k})^2}.
\end{equation}

As in the proof of Proposition~\ref{prop:psteep}, we rewrite all the terms on the right-hand side of \eqref{f0aderivative} using \eqref{defe} as
\begin{multline}\label{hwithea1}
h\left(-\frac{(1-q^\theta)q^k}{1-q^k},s\right)\\
=\left(1+\frac{2(1-q^\theta)q^k(1-q^k)(1-\cos s)}{(1-2q^k+q^{\theta+k})^2}e\left(-\frac{(1-q^\theta)q^k}{1-q^k},s\right)\right)\sin s
\end{multline}
and
\begin{equation}\label{hwithea2}
h\left(1-q^\theta,s\right)=\left(1-\frac{2(1-q^\theta)(1-\cos s)}{(2-q^\theta)^2}e(1-q^\theta,s)\right)\sin s.
\end{equation}
Next we use the fact that the function $f_0^\a$ has a double critical point at $\theta$.
Since the derivative of $h(b,s)$ at $s=0$ is equal to $1$ for all $b$ it implies
that the sum of the coefficients on the right-hand side of \eqref{f0aderivative} in front of the $h$ functions is zero.
As a consequence after substituting \eqref{hwithea1}--\eqref{hwithea2} to the right-hand side of \eqref{f0aderivative}
the main terms not containing the function $e$ exactly cancel.
After removing the terms that cancel and simplifying, we are left with
\begin{equation}\label{f0aderivative2}\begin{aligned}
&-\frac1{(1-\cos s)\sin s}\Re\frac{\d}{\d s}f_0^\a\left(\log_q(1-(1-q^\theta)e^{is})\right)\\
&\quad=\sum_{k=1}^\infty\frac{2(1-q^\theta)^2q^{2k}(1-q^k)^2}{(1-q^{\theta+k})^2(1-2q^k+q^{\theta+k})^2}e\left(-\frac{(1-q^\theta)q^k}{1-q^k},s\right)\\
&\qquad-\kappa_\a L\frac{4(1-q^\theta)^2}{q^{3\theta}(2-q^\theta)}\left[1-\frac{(1-q^\theta)(1-\cos s)}{(2-q^\theta)^2}e(1-q^\theta,s)\right]e(1-q^\theta,s)\\
&\qquad-(f_\a+1)\frac{2(1-q^\theta)^2}{q^{2\theta}(2-q^\theta)^2}e(1-q^\theta,s).
\end{aligned}\end{equation}

In order to complete the argument for the steep descent property,
we take the asymptotic behaviour of the right-hand side of \eqref{f0aderivative2} as $\theta\to0$ for $q,R,L$ fixed.
The first term on the right-hand side of \eqref{f0aderivative2} containing the infinite sum behaves like
a constant multiple of $(1-q^\theta)^2\sim(\log q)^2\theta^2$.
Here we used that the first argument of the $e$ function is negative, hence by Lemma~\ref{lemma:eproperties}, the function is bounded by $1$.
As we will see, this first term is negligible compared to the other terms.

In the second term on the right-hand side of \eqref{f0aderivative2} first we claim that the difference in the parenthesis converges to $1$ uniformly in $s$.
This is because $e(1-q^\theta,s)\le(2-q^\theta)^2q^{-2\theta}\to1$ by Lemma~\ref{lemma:eproperties},
hence the $(1-q^\theta)\to0$ factor in the second term of the parenthesis makes this term negligible.
For the main factor, we claim that
\begin{equation}\label{kappaasymp}
\kappa_\a\sim\frac1{R+L}\frac1{(\log q)^2}\frac1{\theta^2}
\end{equation}
as $\theta\to0$ which can be seen by definition \eqref{defkappaa} and using \eqref{psi'series}.
Hence the coefficient of the $e(1-q^\theta,s)$ function in the second term on the right-hand side of \eqref{f0aderivative2}
converges to $-4L/(R+L)$ as $\theta\to0$.

The asymptotics of the last term on the right-hand side of \eqref{f0aderivative2} can be found similarly.
By \eqref{deffa}, \eqref{kappaasymp} and \eqref{psiseries}, we have that
\begin{equation}
f_\a+1\sim\frac{R-L}{R+L}\frac1{(\log q)^2}\frac1{\theta^2}
\end{equation}
as $\theta\to0$.
Therefore, the coefficient of $e(1-q^\theta,s)$ in the third term on the right-hand side of \eqref{f0aderivative2} goes to $2(L-R)/(R+L)$ as $\theta\to0$.

To summarize we see that the derivative in \eqref{f0aderivative2} as $\theta\to0$ is asymptotically equal to $-2e(1-q^\theta,s)\le-2$
by Lemma~\ref{lemma:eproperties} uniformly in $s$, hence the steep descent property follows for $\theta$ close enough to $0$.

\item
Lemma~\ref{lemma:contourderivatives} with $c=0$, $\gamma=1$ and $r=q^\theta$ and by direct computation we get that
\begin{equation}\label{f0aderivativet}\begin{aligned}
\Re\frac{\d}{\d t}f_0^\p\left(\theta+i\frac t{\log q}\right)
&=\sum_{k=0}^\infty g\left(q^{\theta+k},t\right)-\kappa_\a Rq^\theta\sin t-\kappa_\a Lq^{-\theta}\sin t\\
&=\sum_{k=0}^\infty\frac{q^{\theta+k}}{(1-q^{\theta+k})^2}h(q^{\theta+k,t})-\kappa_\a(Rq^\theta+Lq^{-\theta})h(0,t)
\end{aligned}\end{equation}
by writing the derivative in terms of the function $h$ using \eqref{defh}.
By Lemma~\ref{lemma:hcompare}, $h(q^{\theta+k},t)\le h(0,t)$ for all $t\in(0,\pi)$
which can be used to upper bound the first sum on the right-hand side of \eqref{f0aderivativet}.
By the fact that $\theta$ is a double critical point for the function $f_0^\a$,
the sum of the coefficients on the right-hand side of \eqref{f0aderivativet} is zero
which can also be seen directly by \eqref{defkappaa} and \eqref{psi'series}.
Hence the bound $h(q^{\theta+k},t)\le h(0,t)$ implies that the derivative \eqref{f0aderivativet} is non-positive for $t\in(0,\pi)$
which yields the steep descent property.
\end{enumerate}
\end{proof}



\bibliography{Biblio}
\bibliographystyle{alpha}
\end{document}